\documentclass[11pt]{amsart}

\usepackage[T1]{fontenc}
\usepackage[utf8]{inputenc}
\usepackage{amssymb,amsmath,amsfonts}
\usepackage{amssymb,fullpage,amsthm}
\usepackage{mathtools}

\usepackage[colorlinks=true,linkcolor=blue,citecolor=red]{hyperref}

\usepackage{cite}
\usepackage{color}

\newcommand{\2}{{L^2\left(\mathbb{R}^3\right)}}

\newcommand{\Linfty}{L^\infty\left(\mathbb{R}^3\right)}
\DeclareMathOperator{\dive}{{div\hspace{0.2mm}}}

\newcommand{\lee}{\lambda_{\epsilon_1,\epsilon_2}}

\newcommand{\R}{\mathbb{R}}

\newcommand{\sumf}{\sum_{\left| q-q' \right|\leqslant4}}
\newcommand{\sumi}{\sum_{q'>q-4}}
\newcommand{\tq}{\Delta_q}
\newcommand{\Sq}{S_q}

\newcommand{\hra}{\hookrightarrow}

\newcommand{\T}{\mathbb{T}}
\newcommand{\B}{\mathcal{B}}
\newcommand{\ub}{\begin{pmatrix} u^\varepsilon \\ b^\varepsilon \end{pmatrix}}

\newcommand{\Prr}{\mathcal{P}_{r,R}}
\newcommand{\Crr}{\mathcal{C}_{r,R}}
\newcommand{\tU}{\widetilde{U}^\varepsilon}
\newcommand{\bU}{\overline{U}^\varepsilon}

\newcommand{\Hs}{{H^s\left( \R^3 \right)}}
\newcommand{\Woi}{W^{1,\infty}\left( \R^3 \right)}

\newcommand{\qg}{{quasi-geostrophic }}
\newcommand{\NS}{Navier-Stokes }

\newcommand{\eps}{\varepsilon}

\newcommand{\bgam}{\overline{\gamma}}

\newcommand{\NN}{\mathbb{N}}

\newcommand{\RR}{\mathbb{R}}

\newcommand{\dd}{\partial}

\newcommand{\abs}[1]{\left\vert#1\right\vert}
\newcommand{\set}[1]{\left\{#1\right\}}
\newcommand{\psca}[1]{\left\langle#1\right\rangle}
\newcommand{\pint}[1]{\left[#1\right]}
\newcommand{\pare}[1]{\left(#1\right)}
\newcommand{\norm}[1]{\left\Vert#1\right\Vert}

\newcommand{\spres}{\sum_{\abs{q'-q}\leqslant 4}}
\newcommand{\sloin}{\sum_{q'> q-4}}

\newcommand{\ds}{\displaystyle}
\theoremstyle{theorem}
\newtheorem{theorem}{Theorem}[section]
\newtheorem{prop}[theorem]{Proposition}
\newtheorem{lemma}[theorem]{Lemma}

\newtheorem{defn}[theorem]{Definition}
\theoremstyle{definition}

\newtheorem{rem}[theorem]{Remark}

\numberwithin{equation}{section}

\begin{document}

\title{Dispersive effects of weakly compressible and fast rotating inviscid fluids}

\author[V.-S. Ngo \& S. Scrobogna]{Van-Sang Ngo \& Stefano Scrobogna}

\thanks{The research of the second author was partially supported by the Basque Government through the BERC 2014-2017 program and by the Spanish Ministry of Economy and Competitiveness MINECO: BCAM Severo Ochoa accreditation SEV-2013-0323.}

\date{\today}

\address{\noindent \textsc{V.-S. Ngo,  Universit\'e de Rouen, Laboratoire de Math\'ematiques Rapha\"el Salem, UMR 6085 CNRS, 76801 Saint-Etienne du Rouvray, France}}
\email{van-sang.ngo@univ-rouen.fr}

\address{\noindent \textsc{S. Scrobogna, IMB, Universit\'e de Bordeaux, 351, cours de la Lib\'eration, 33405 Talence, France, \& Basque Center for Applied Mathematics, Mazarredo, 14,  E48009 Bilbao, Basque Country – Spain}}
\email{stefano.scrobogna@math.u-bordeaux1.fr}
 

\begin{abstract}
	We consider a system describing the motion of an isentropic, inviscid, weakly compressible, fast rotating fluid in the whole space $\RR^3$, with initial data belonging to $ H^s \left( \R^3 \right), s>5/2 $. We prove that the system admits a unique local strong solution in $ L^\infty \left( [0,T]; H^s\left( \R^3 \right) \right) $, where $ T $ is independent of the Rossby and Mach numbers. Moreover, using Strichartz-type estimates, we prove the longtime existence of the solution, \emph{i.e.} its lifespan is of the order of $\varepsilon^{-\alpha}, \alpha >0$, without any smallness assumption on the initial data (the initial data can even go to infinity in some sense), provided that the rotation is fast enough. 
\end{abstract}

\maketitle

\section{Introduction}

In this paper, we consider the following system of weakly compressible, fast rotating fluids in the whole space $\RR^3$
\begin{equation}
	\tag{CRE$_{\varepsilon,\theta}$}
	\label{compressible_rotating_E}
	\left\lbrace
	\begin{aligned}
		&\partial_t\left(  \rho^{\varepsilon,\theta} u^{\varepsilon,\theta} \right) + \dive \left(\rho^{\varepsilon,\theta} u^{\varepsilon,\theta} \otimes u^{\varepsilon,\theta} \right)  + \frac{1}{\theta^2}\nabla P\left( \rho^{\varepsilon,\theta} \right) + \frac{1}{\varepsilon} e^3\wedge \left(  \rho^{\varepsilon,\theta} u^{\varepsilon,\theta} \right) =0\\
		&\partial_t \rho^{\varepsilon,\theta} + \dive  \left( \rho^{\varepsilon,\theta} u^{\varepsilon,\theta} \right) = 0\\
		&\left. \left( \rho^{\varepsilon,\theta}, u^{\varepsilon,\theta} \right)\right|_{t=0}= \left( \rho^{\varepsilon,\theta}_0, u^{\varepsilon,\theta}_0 \right).
	\end{aligned}
	\right.
\end{equation}
Here, the Rossby number $\eps$ represents the ratio of the displacement due to inertia to the displacement due to Coriolis force. On a planetary scale, the displacement due to inertial forces, \emph{i.e.} the collision of air molecules (in the case of the atmosphere) or water  molecules (in the case of oceans) is generally much smaller than the relative displacement due to the rotation of the Earth around his own axis. Away from persistent streams such as the Gulf stream, the value of Rossby number is around $10^{-3}$. On the other hand, the Mach number is a dimensionless number representing the ratio between the local flow velocity and the speed of sound in the medium. For geophysical fluids appearing in meteorology for exemple, the Mach number $\theta$ is also very small.

We want to have a few words about the low Mach-number regime and the fast rotation limit. For the incompressible limit, the fluid is expected to have an incompressible behaviour. In the fast rotation limit, the Coriolis force becomes dominant and plays a very important role. Indeed, the fast rotating fluid have tendency to stabilize and to move in vertical columns (the so-called ``Taylor-Proudman'' columns). This phenomenon can be observed in many geophysical fluids (such as oceanic currents in the western North Atlantic) and is well known in fluid mechanics as the Taylor-Proudman theorem (see \cite{Pedlosky87} for more details).

We remark that if $\eps \ll \theta$ or $\eps \gg \theta$, then either the high rotation or the weak compressibility dominates the other, and one can separately take the high rotation limit and the weak compressible limit. In this paper, we are interested in the case where these two numbers are very small and where the high rotation and weak compressibility limits occur at the same scale, \emph{i.e.} $\theta = \eps \rightarrow 0$. Moreover, in our study, we suppose that the fluid is inviscid and isentropic, which means that it has no viscosity and the pressure satisfies $$P = P(\rho) = A \rho^\gamma,$$ where $A > 0$ and $\gamma > 1$ are given. We refer the reader to \cite{Pedlosky87} and the references therein for further physical explanations, and to \cite{CDGGbook} for a brief physical introduction of fast rotating hydrodynamic systems with strong emphasis on the problem under the mathematical point of view.

\subsection{Weakly compressible, fast rotating fluids and related systems}

Let us give a brief introduction of systems related to \eqref{compressible_rotating_E}. In general, the motion of a compressible fluid with a homogeneous temperature can be derived from the laws of conservation of mass and of linear momentum (see \cite{Bat99}, \cite{LanLifs} or \cite{LionsPL1} for instance), and is described by the following system 
\begin{equation*}
	\left\lbrace
	\begin{aligned}
		& \partial_t \left( \rho u \right) + \dive \left( \rho u \otimes u \right)- \dive \left( \sigma \right)  = \rho f\\
		&\partial_t\rho + \dive \left( \rho u \right)=0.
	\end{aligned}
	\right.
\end{equation*} 
Here, $\sigma$ is the \emph{stress tensor} and $f$ represents the external body forces acting on the fluid (gravity, Coriolis, electromagnetic forces, etc\ldots). For an isotropic \emph{newtonian fluid}, the stress tensor is supposed to be linearly dependent on the strain rate tensor $D = \frac{1}2\pare{\nabla u+ {}^T\nabla u}$, and writes
\begin{equation*}
	\sigma = -p \mathbf{1} + \lambda \dive u +  \mu \pare{\nabla u+ {}^T\nabla u},
\end{equation*}
where the scalar function $p$ stands for the \emph{pressure}, $\mathbf{1}$ is the identity matrix (tensor) and $\mu, \lambda$ are the Lam\'e viscosity coefficients (which may depend on the density $\rho$), $\mu > 0$ and $\mu + \lambda > 0$. In fluid mechanics, $\mu$ is referred to as the dynamic viscosity of the fluid and in a case of a barotropic fluid, $p$ is a function of the density $\rho$ only. These considerations lead to the following system describing the motion of a compressible newtonian barotropic fluid
\begin{equation}
	\tag{CNS}\label{eq:compressible_generic}
	\left\lbrace
	\begin{aligned}
		& \partial_t \left( \rho u \right) + \dive \left( \rho u \otimes u \right)- \dive \left( \lambda \dive u +  \mu \left( \nabla u+ {}^T \nabla u \right) \right) + \nabla p  = \rho f\\
		&\partial_t\rho + \dive \left( \rho u \right)=0.
	\end{aligned}
	\right.
\end{equation}

\noindent In the case of low Mach-number flow, the Mach number $\theta$ is suppose to be very small (the fluid is pseudo-incompressible), let
\begin{equation*}
	\rho^\theta \left( t,x \right)= \rho \left( \frac{t}{\theta}, x \right) \quad \mbox{and} \quad u^\theta \left( t,x \right) = \frac{1}{\theta} \; u \left( \frac{t}{\theta} , x \right),
\end{equation*}
so the system \eqref{eq:compressible_generic}, endowed with some initial data, becomes
\begin{equation}
	\tag{CNS$_\theta$}\label{compressible_nonrotating_NS}
	\left\lbrace
	\begin{aligned}
		&\partial_t \left(\rho^\theta u^\theta \right) + \dive \left(\rho^\theta u^\theta \otimes u^\theta \right) - \mu \Delta u^\theta - \lambda \nabla \dive u^\theta  + \frac{1}{\theta^2}\nabla P\big( \rho^\theta \big)  =\rho^\theta f\\
		&\partial_t \rho^\theta + \dive  \left( \rho^\theta u^\theta \right) =0\\
		&\left. \left( \rho^\theta, u^\theta \right)\right|_{t=0}= \left( \rho^\theta_0, u^\theta_0 \right).
	\end{aligned}
	\right.
\end{equation}
In the inviscid case where $\lambda=\mu=0$ and we obtain the following system 
\begin{equation}
	\tag{CE$_\theta$}\label{compressible_nonrotating_E}
	\left\lbrace
	\begin{aligned}
		&\partial_t\left(  \rho^\theta u^\theta \right) + \dive \left(\rho^\theta u^\theta \otimes u^\theta \right)  + \frac{1}{\theta^2}\nabla P\big( \rho^\theta \big) = \rho^\theta f\\
		&\partial_t \rho^\theta + \dive  \left( \rho^\theta u^\theta \right) =0\\
		&\left. \left( \rho^\theta, u^\theta \right)\right|_{t=0}= \left( \rho^\theta_0, u^\theta_0 \right).
	\end{aligned}
	\right.
\end{equation}

Now, for geophysical fluids such as the oceans or the atmosphere, effects of the rotation of the Earth can not be neglected. Rewriting the systems \eqref{compressible_nonrotating_NS} or \eqref{compressible_nonrotating_E} in a rotating frame of reference tied to the Earth, we have to take into account two factors, the Coriolis acceleration and the centrifugal acceleration. We assume that the centrifugal force is in equilibium with the stratification due to the gravity of the Earth, and so can be neglected. We also suppose that the rotation axis is parallel to the $ x_3 $-axis, and that the speed of rotation is constant, which is often considered in the study of geophysical fluids in mid-latitude regions. Then, the system \eqref{compressible_nonrotating_NS} writes
\begin{equation}
	\tag{CRNS$_{\varepsilon,\theta}$}\label{compressible_rotating_NS}
	\left\lbrace
	\begin{aligned}
		&\partial_t\left(  \rho^{\varepsilon,\theta} u^{\varepsilon,\theta} \right) + \dive \left(\rho^{\varepsilon,\theta} u^{\varepsilon,\theta} \otimes u^{\varepsilon,\theta} \right) - \mu \Delta u^{\varepsilon,\theta} - \lambda \nabla \dive u^{\varepsilon,\theta}\\
		&\hspace{7cm}  + \frac{1}{\theta^2}\nabla P\left( \rho^{\varepsilon,\theta} \right) + \frac{1}{\varepsilon} e^3\wedge \left(  \rho^{\varepsilon,\theta} u^{\varepsilon,\theta} \right) =0\\
		&\partial_t \rho^{\varepsilon,\theta} + \dive  \left( \rho^{\varepsilon,\theta} u^{\varepsilon,\theta} \right) =0\\
		&\left. \left( \rho^{\varepsilon,\theta}, u^{\varepsilon,\theta} \right)\right|_{t=0}= \left( \rho^{\varepsilon,\theta}_0, u^{\varepsilon,\theta}_0 \right).
	\end{aligned}
	\right.
\end{equation}
In the case where there is no viscosity, we obtain the system \eqref{compressible_rotating_E}.

\subsection{Brief recall of known results}

For non-rotating fluids, many results have been obtained concerning the systems \eqref{compressible_nonrotating_NS} and \eqref{compressible_nonrotating_E} in the case of \textit{well prepared } initial data, \emph{i.e.} $$\rho_0^\theta= 1+ \mathcal{O}\left( \theta^2 \right) \quad\mbox{and}\quad \dive u_0^\theta=\mathcal{O}\left( \theta \right),$$ for which, we refer to the works \cite{HoffMach}, \cite{KlMaj}, \cite{KLN} or \cite{LinPhD}. In the case of \textit{ill prepared} initial data, it is only assumed that $$\rho_0^\theta = 1 + \theta b^\theta_0$$ and $\left( b_0^\theta, u_0^\theta \right)$ are only bounded in some suitable spaces which does not necessarily belong to the kernel of the penalized operator. If $\mathbb{P}u_0^\theta \to v_0$ when $\theta$ goes to zero\footnote{Here $ \mathbb{P} $ is the Leray  projector on the space of solenoidal vector fields defined as $ \mathbb{P}= I- \Delta^{-1}\nabla\dive  $}, one expects that $u^\theta \to v$ where $v$ is the solution of the incompressible \NS equations 
\begin{equation}
	\tag{INS}\label{incompressible_NS}
	\left\lbrace
	\begin{array}{l}
		\partial_t v + v \cdot \nabla v -\mu \Delta v + \nabla \Pi =0,\\
		\dive v=0\\
		\left. v\right|_{t=0}=v_0.
	\end{array}
	\right.
\end{equation}
The expected convergence is however not easy to be rigorously justified. The main difficulty lies in the fact that one has to deal with the propagation of acoustic waves with speed of order $\theta^{-1}$, a phenomenon which does not occur in the case of well prepared data.

In \cite{LionsPL2}, P.-L- Lions proved the existence of global weak solutions of \eqref{compressible_nonrotating_NS}. The fluid is supposed to be isentropic and the pressure is of the form $P\left( \rho \right)= a\rho^\gamma$, with certain restrictions on $ \gamma $ depending on the space dimension $d$. In the same setting P.-L. Lions and N. Masmoudi in \cite{LM} proved that weak solutions of \eqref{compressible_nonrotating_NS} converges weakly to weak solutions of \eqref{incompressible_NS} in various boundary settings. This result is proved via some weak compactness methods (see also \cite{IGLSR} and \cite{Fanelli2}). In the work of B. Desjardins, E. Grenier, P.-L. Lions and N. Masmoudi \cite{DGLM}, considering \eqref{compressible_nonrotating_NS} with $ f\equiv 0 $, in a bounded domain $\Omega$ with Dirichlet boundary conditions, the authors proved that as  $\theta \to 0$, the global weak solutions of \eqref{compressible_nonrotating_NS} converge weakly in $L^2$ to a global weak solution of the incompressible \NS equations \eqref{incompressible_NS}. In \cite{DG99}, using dispersive Strichartz-type estimates, B. Desjardins and E. Grenier proved that the gradient part of the velocity field (\emph{i.e.} the gradient of the acoustic potential) of the system \eqref{compressible_nonrotating_NS} converges strongly to zero. We also refer to the works of E. Feireisl and H. Petzeltová (see for example \cite{FP1}, \cite{FP2}) concerning the compactness properties and the longtime behaviour of weak solutions of \eqref{compressible_nonrotating_NS}. Finally, we want to mention the works of R. Danchin \cite{Danchin00} and \cite{Danchin02}. In \cite{Danchin00}, the author proved global existence of \textit{strong} solutions for the system \eqref{compressible_nonrotating_NS} for small initial data in some suitable, critical, scale-invariant (Besov) spaces, in the same spirit as in the work of M. Cannone, Y. Meyer and F. Planchon \cite{CMP} or the work of H. Fujita and T. Kato \cite{FK} for the incompressible model. In \cite{Danchin02}, R. Danchin addressed the convergence of \eqref{compressible_nonrotating_NS} to \eqref{incompressible_NS} for ill-prepared initial data when the Mach number $\theta$ tends to zero. When the initial data are small, the author obtains global convergence and existence, while for large initial data with some further regularity assumptions, it is shown that the solution of \eqref{compressible_nonrotating_NS} exists and converges to the solution of \eqref{incompressible_NS} in the same time interval of existence of the solution of \eqref{incompressible_NS}. For compressible inviscid fluids in the non-rotating case, in A. Dutrifoy and T. Hmidi \cite{DH03}, the authors considered the system \eqref{compressible_nonrotating_E} in $ \R^2 $ with initial data not uniformly smooth (\emph{i.e.} the $\mathcal{C}^1$ norm is of order $ \mathcal{O}\left( \theta^{-\alpha} \right), \alpha >0 $). The convergence to strong, global solutions of 2D Euler equation is proved by mean of Strichartz estimates and the propagation of the minimal regularity. In this incompressible limit context, we also mention the work of G. M\'etivier and S. Schochet \cite{MS}.

%\bigskip

The literature concerning incompressible fast rotating fluids is broad. Here, we only restrict ourselves to mention some important previous works on the whole space case. For a complete survey of the whole space case and the other cases, we send to the book \cite{CDGGbook} and the references therein. We first recall the works of J.-Y. Chemin, B. Desjardins, I. Gallagher and E. Grenier \cite{CDGG} and \cite{CDGG2} for incompressible viscous rotating fluids, with initial data of the form $$u_0 = \overline{u}_0 + \widetilde{u}_0,$$ where the 2D part $\overline{u}_0$ only depends on $(x_1,x_2)$ and the 3D part $\widetilde{u}_0$ belongs to the anisotropic Sobolev spaces $H^{0,s}$, with $s > \frac{1}2$. It is proved that the 2D part is governed by a 2D incompressible Navier-Stokes system, while the 3D part converges to zero as the Rossby number $\eps \to 0$, using Strichartz estimates obtained for the associated linear free-wave system. As a consequence, if the rotation is fast enough, the solution of the 3D incompressible viscous rotating fluids exists globally in time and converges to the solution of the 2D incompressible Navier-Stokes system. In the case of incompressible, inviscid fluids, however, we cannot get the global existence of strong solutions when the rotation is fast, due to the lack of smoothing effect given by the viscous term. It is proved in A. Dutrifoy \cite{Dutrifoy2} that if the rotation is fast enough ($\eps \to 0$), the solution of an incompressible inviscid rotating fluid can exist in long time intervals of size at least equivalent to $\ln \ln \eps^{-1}$. However, in the case where the viscosity is not zero, but very small (of order $\eps^\alpha$, for $\alpha$ in some interval $[0,\alpha_0[$), when the rotation is fast enough, the global existence of strong solutions can still be proven in the case of pure 3D initial data (see \cite{VSN}).

Let us now focus on fast rotating, compressible fluids. To the best of our knowledge, there is no result yet concerning the inviscid system \eqref{compressible_rotating_E}. In the viscous fast rotating case, in \cite{FGN}, E. Fereisl, I. Gallagher and A. Novotn\'y studied the dynamics, when $\theta = \varepsilon \to 0 $, of weaks solutions of the system \eqref{compressible_rotating_NS} in $\R^2\times ]0,1[$, with non-slip boundary conditions 
\begin{align*}
	\left. u^{\varepsilon, 3}\right|_{x_3=0,1} = 0 \quad \mbox{and} \quad
	\left. \left( S_{2,3}, -S_{1,3},0 \right)\right|_{x_3=0,1}=&0,
\end{align*}
where $ \mathbb{S} $ is the viscous stress tensor 
$$
\mathbb{S}\left( \nabla u \right)= \mu \left( \nabla u + {}^T\nabla u -\frac{2}{3}\dive u I \right).
$$
Their result relies on the spectral analysis of the singular perturbation operator. Using RAGE theorem (see \cite{RS3}), the authors proved the dispersion due to fast rotation and that weak solutions of \eqref{compressible_rotating_NS} converges to a 2D viscous \qg equation for the limit density. We refer to \cite{FGN} for a detailed description of the limit system. In \cite{FGG-VN}, Feireisl, Gallagher, Gérard-Varet and Novotný studied the system \eqref{compressible_rotating_NS} in the case where the effect of the centrifugal force was taken into account. Noticing that this term scales as $\varepsilon^{-2}$, they studied both the isotropic limit and the multi-scale limit: namely, they supposed the Mach-number to be proportional to $\varepsilon^m $, for $m\geqslant 1$. We want to point out that, in the analysis of the isotropic scaling ($m=1$), the authors had to resort to compensated compactness arguments in order to pass to the limit: as a matter of fact, the singular perturbation operator had variable coefficients, and spectral analysis tools were no more available. Recently in \cite{Fanelli1}, F. Fanelli proved a similar result as the one proved in \cite{FGN}, by adding to the system \eqref{compressible_rotating_NS} a capillarity term and studying various regimes depending on some positive parameter.

To complete our brief survey of known results, we want to remark that all the compressible systems previously mentioned are isothermal. In the case of variable temperature, the generic system governing a heat conductive, compressible fluid is the following 
\begin{equation}
	\label{compressible_heat_conductive_NS}
	\tag{HCCNS}
	\left\lbrace
	\begin{aligned}
		&\partial_t \rho + \dive \left( \rho u \right)=0,\\
		& \partial_t \left( \rho u \right) + \dive \left( \rho u \otimes u \right)- \dive \left( \tau \right) + \nabla P = \rho f,\\
		&\partial_t \left( \rho \left( \frac{\left| u \right|^2}{2} + e \right) \right) + \dive \left[ u\left( \rho \left( \frac{\left| u \right|^2}{2} + e \right) + P \right) \right]= \dive \left( \tau \cdot u \right) - \dive q +\rho f \cdot u,
	\end{aligned}
	\right.
\end{equation}
which can be derived from the conservation of mass, linear momentum and energy. We refer the reader to \cite{LionsPL1} and references therein for more details. Here, the fluid is always supposed to be newtonian and $e=e\left( t,x \right)$ is the internal (thermal) energy per unit mass. The heat conduction $ q $ is given by $ q= -k \nabla\mathcal{T} $, where $ k $ is positive and $\mathcal{T} $ stands for the temperature. If $ e $ obeys Joule rule (\emph{i.e.} $ e $ is a function of $ \mathcal{T} $ only), the initial data is smooth and the initial density is bounded and bounded away from zero, the existence and uniqueness of a local classical solution has already been known for a long time (see \cite{Itaya} or \cite{Nash}). In \cite{Danchin01}, for $ p\in [1,\infty[$, R. Danchin proved that \eqref{compressible_heat_conductive_NS} admits a local (in time) solution in the critical scale-invariant space $ B^{\frac{N}{p}-1}_{p,1}\left( \R^N \right) $ if $p < N$ and the solution is unique if $p < 2N/3$.

\subsection{Main result and structure of the paper.}

The aim of this paper is to study the behavior of \textit{strong} solutions of the system \eqref{compressible_rotating_E} in the limit $\theta = \varepsilon\to 0 $ and in the case of ill-prepared initial data in the whole space $\R^3$, say
$$\rho_0^\eps = 1 + \eps b_0^\eps.$$ Let $\overline{\gamma} = (\gamma - 1)/2$. We consider the substitution  
$$
1+ \varepsilon b^\varepsilon = \frac{\left( 4\gamma A \right)^{1/2}}{\gamma-1} \left( \rho^\varepsilon \right)^{\overline{\gamma}}
$$ 
and \eqref{compressible_rotating_E} becomes (after a few algebraic calculations)
\begin{equation}\label{weakly_compressible_E}
	\left\lbrace
	\begin{aligned}
		&\partial_t \ub -\frac{1}{\varepsilon}\mathcal{ B} \ub +
		\begin{pmatrix}
			u^\varepsilon \cdot \nabla u^\varepsilon + \bgam \, b^\eps \nabla b^\eps\\
			u^\varepsilon \cdot \nabla b^\varepsilon + \bgam \, b^\varepsilon \dive u^\varepsilon
		\end{pmatrix} = 0,\\ 
		&\left. \left( u^\varepsilon, b^\varepsilon \right)\right|_{t=0}= \left( u_0^\eps, b_0^\eps \right),
	\end{aligned}
	\right.
\end{equation}
where $\B$ is the following operator
\begin{align}\label{definizione_B}
	\B = \left( \begin{array}{cccc}
		0&1&0& -\bgam\partial_1\\
		-1 & 0 & 0 & - \bgam\partial_2\\
		0&0&0 & - \bgam\partial_3\\
		- \bgam\partial_1 & -\bgam\partial_2 & -\bgam\partial_3 & 0
	\end{array} \right),
\end{align}
and where $\partial_i$, for any $i \in \set{1,2,3}$ stands for the derivative with respect to $x_i$ variable. Moreover we can write the nonlinearity as follows
\begin{equation}
	\label{eq:defA}
	\begin{pmatrix}
		u^\varepsilon \cdot \nabla u^\varepsilon + \bgam \ b^\eps \nabla b^\eps\\
		u^\varepsilon \cdot \nabla b^\varepsilon + \bgam \ b^\varepsilon \dive u^\varepsilon
	\end{pmatrix}
	= \mathcal{A} \left( U^\eps , D \right) U^\eps = 
	\begin{pmatrix}
		u^\varepsilon \cdot\nabla & 0 & 0 & \bgam b^\varepsilon \partial_1\\
		0 & u^\varepsilon\cdot\nabla & 0 & \bgam b^\varepsilon \partial_2\\
		0& 0 & u^\varepsilon\cdot\nabla  & \bgam b^\varepsilon \partial_3\\
		\bgam b^\varepsilon \partial_1 & \bgam b^\varepsilon \partial_2 & \bgam b^\varepsilon \partial_3 &  u^\varepsilon\cdot\nabla 
	\end{pmatrix} \ub,
\end{equation}
where $U^\eps$ stays for $ \ub $. With all the above considerations, the system \eqref{weakly_compressible_E} can be rewritten as 
\begin{equation}\label{eq:WCEshortsym}
\left\lbrace
\begin{aligned}
&\partial_t U^\eps -\frac{1}{\varepsilon}\mathcal{ B} U^\eps + \mathcal{A} (U^\eps,D) U^\eps = 0,\\ 
& \left. U^\eps\right|_{t=0}= U_0^\eps= \left( u_0^\eps, b_0^\eps \right).
\end{aligned}
\right.
\end{equation}

\begin{rem}
We would like to underline that, given a $ \2 $ vector field $ F $, we have
$$
\left(\left. \B F \right|  F \right)_\2= \left(\left. \widehat{\B F} \right|  \widehat{F} \right)_\2= 0.
$$
\end{rem}

In order to state our result, we recall the definitions of the functional spaces we will use in our paper. We use the index ``h'' to refer to the horizontal variable, and the index ``v'' or ``3'' to refer to the vertical one. Thus, $x_h = (x_1,x_2)$ and $\xi_h = (\xi_1,\xi_2)$. The anisotropic Lebesgue spaces $L_h^pL_v^q$ with $p, q \geq 1$ are defined as
\begin{align*}
	L_h^pL_v^q(\RR^3) &= L^p(\RR^2_h;L^q_v(\RR)) = \Big\{u\in \mathcal{S'}: \norm{u}_{L_h^pL_v^q} = \Big[ \int_{\RR^2_h} \Big\vert \int_{\RR_v}\abs{u(x_h,x_3)}^qdx_3 \Big\vert^{\frac{p}q}dx_h \Big]^{\frac{1}p} < +\infty\Big\}.
\end{align*}
Here, the order of integration is important. Indeed, if $1 \leq p \leq q$ and if $u: X_1\times X_2 \rightarrow \RR$ is a
function in $L^p(X_1;L^q(X_2))$, where $(X_1,d\mu_1)$, $(X_2,d\mu_2)$ are measurable spaces, then $u \in L^q(X_2;L^p(X_1))$ and
\begin{equation*}
	\norm{u}_{L^q(X_2;L^p(X_1))} \leq \norm{u}_{L^p(X_1;L^q(X_2))}.
\end{equation*}
We recall that the non-homogeneous Sobolev spaces $\Hs$, with $s \in \RR$, are defined as the closure of the set of rapidly decreasing smooth functions under the norm
\begin{equation*}
	\norm{u}_{H^s} \stackrel{\mbox{\tiny def}}{=} \pare{\int_{\RR^3} \pare{1 + \abs{\xi}^2}^s\abs{\widehat{u}(\xi)}^2 d\xi}^{\frac{1}2},
\end{equation*}
where $\widehat{u}$ is the usual Fourier transform of $u$. For any $s \geqslant 0$, $s_0>0$, $1 < p < 2$, we define the spaces 
\begin{equation}
	\label{eq:DefY} Y_{s,s_0,p} = H^{s+s_0}\left( \R^3 \right)^4 \cap L^2_hL^p_v\pare{\RR^3}^4 \cap L^p_hL^2_v\pare{\RR^3}^4,
\end{equation}
endowed with the norm
\begin{equation}
	\label{eq:normY} \norm{u}_{s,s_0,p} = \max\set{\norm{u}_{H^{s+s_0}}, \norm{u}_{L^2_hL^p_v}, \norm{u}_{L^p_hL^2_v}}.
\end{equation}
From now on, for any initial data $U_0^\eps \in Y_{s,s_0,p}$, we set 
\begin{equation}
	\label{eq:C0} \mathcal{C}(U_0^\eps) = \max\set{\norm{U_0^\eps}_{s,s_0,p}, \norm{U_0^\eps}_{s,s_0,p}^2}.
\end{equation}

The main result of this paper is the following theorem.
\begin{theorem}
	\label{th:mainth}
	Let $s>5/2$, $s_0>0$ be fixed constants, $1 < p < 2$. For any $\eps > 0$ and for any $U_0^\eps \in Y_{s,s_0,p}$ there exist a time $T^\star_\varepsilon > 0$ and a unique solution $U^\varepsilon = \left( u^\varepsilon, b^\varepsilon \right) $ of system \eqref{weakly_compressible_E} with initial data $U_0^\eps$, satisfying $U^\varepsilon \in C\pare{[0,T^\star_\varepsilon]; \Hs}$. Moreover, there exist positive constants $\;\overline{C} > 0$, $\alpha > 0$ and $\eps^\star > 0$ such that, for any $\eps \in ]0,\eps^\star]$,
	\begin{equation*}
		T^\star_\varepsilon \geqslant \frac{\overline{C}}{\mathcal{C}(U_0^\eps) \, \varepsilon^{\alpha}},
	\end{equation*}
	where  $\mathcal{C}(U_0^\eps)$ is defined in \eqref{eq:C0}.	
\end{theorem}

\begin{rem}
	\mbox{}
	\begin{enumerate}
		\item The estimate of the lifespan $T^\star_\varepsilon$ of $U^\eps$ is much better than in \cite{Dutrifoy2} (for incompressible fast rotating fluids). The reason is that we only consider 3D initial data of finite energy in $\RR^3$. As a consequence, the limit system is zero, since the only vector field of finite energy in $\RR^3$ which belongs to the kernel of the penalized operator $\mathcal{\B}$ is zero. In the more general case where the initial data is the sum of a 2D part (which belongs to the kernel of the penalization operator $\mathcal{B}$) and a 3D part (of finite energy in $\RR^3$), the limit system is not zero but some 2D nonlinear hyperbolic system. Thus, in the case of general data, we can only hope for a similar lifespan as in \cite{Dutrifoy2}. This general case will be dealt with in a forthcoming paper.
		
		\bigskip		
	
		\item For fixed $\eps > 0$ and for small $U_0^\eps$, the lifespan is inversely proportional to the $Y_{s,s_0,p}$-norm of $U_0^\eps$, which is expected for this type of hyperbolic system with small initial data.

		\bigskip

		\item \label{re:13blowupdata} The initial data can be chosen not only to be large but to blow up as $\eps \to 0$. Indeed, for data $U_0^\eps \sim \eps^{-\omega}$, with $0 < \omega < \frac{\alpha}2$, the maximal lifetime of the solution still goes to $\infty$ as
		$$ T^\star_\varepsilon \gtrsim \eps^{-(\alpha - 2\omega)} \rightarrow \infty.$$
	\end{enumerate}
\end{rem}

\bigskip

Throughout this paper, we set 
\begin{equation}
	\label{eq:CrR}
	\Crr= \set{\xi\in\R^3_\xi \;\big\vert\; \abs{\xi} \leqslant R, \abs{\xi_h} \geqslant r, \abs{\xi_3} \geqslant r}. 
\end{equation}
Our strategy to study the system \eqref{weakly_compressible_E} consists in finding a solution to \eqref{weakly_compressible_E} of the form 
$$U^\varepsilon= \left( u^\varepsilon, b^\varepsilon \right)= \bU+ \tU$$ 
where $ \bU= \left( \overline{u}^\varepsilon, \overline{b}^\varepsilon \right) $ and $ \tU= \left( \widetilde{u}^\varepsilon, \widetilde{b}^\varepsilon \right) $ are  respectively solutions to the following systems
\begin{align*}
	&\left\lbrace
	\begin{aligned}
		&\partial_t \bU - \frac{1}{\varepsilon}\B \bU = 0\\
		&\left. \bU \right|_{t=0}= \Psi_{r,R} \left( D \right) \left( u_0^\eps, b_0^\eps \right)
	\end{aligned}\right. ,
	& \left\lbrace
	\begin{aligned}
		&\partial_t \tU - \frac{1}{\varepsilon}\B \tU + \mathcal{A} (U,D) U = 0\\
		&\left. \tU \right|_{t=0}= \left(1- \Psi_{r,R} \left( D \right) \right) \left( u_0^\eps, b_0^\eps \right)
	\end{aligned}
	\right. .
\end{align*}
Here, the frequency cut-off radii $0 < r < R$ will be precisely chosen, depending on $\varepsilon$ and $\Psi_{r,R}$ is a radial function supported in $\mathcal{C}_{\frac{r}2, 2R}$ and is identically equal to 1 in $\Crr$. The precise definition of $\Psi_{r,R}$ will be given in \eqref{eq:PsirR} in Section \ref{se:Data}. We will also prove in Section \ref{se:Data} that, if $R$ is sufficiently large, the system describing $\tU$ can be considered as a 3D hydrodynamical system with small initial data, which is known to be globally well posed in critical spaces (see \cite{CMP}, \cite{Danchin00}, \cite{FK}, \cite{KochTatNS}). For the linear part $\bU$ which describes the evolution of 3D free waves, we will prove that it goes to zero in some appropriate topology using similar Strichartz-type estimates as in \cite{CDGG}, \cite{CDGG2}, \cite{Dutrifoy2} or \cite{VSN}. We want to emphasize that, unlike the RAGE theorem used in \cite{FGN}, Strichartz estimates give very precise quantitative estimates of the rate of decay to zero of $\bU$, as $\eps \to 0$.

This paper will be organized as follows. In Section \ref{se:dyadic} we introduce the notation and a detailed description of the critical spaces that we are going to use all along the work. Moreover, we introduce some elements of the Littlewood-Paley and the paradifferential calculus, which are primordial for the study of critical behavior of nonlinearities. In Section \ref{se:Data}, we study a specific decomposition of the initial data in two parts, one only containing medium Fourier frequencies and the other very high or very low frequencies and we provide a precise control of the latter. Section \ref{se:Strichartz} is devoted to the study of the cut-off  linear free-wave system associate to \eqref{weakly_compressible_E}. Using the spectral properties of the operator $ \B $ defined in \eqref{definizione_B}, we prove some Strichartz-type estimates for this system, which show that its solutions vanish in some appropriate $L^p\pare{\R_+; L^q\pare{\R^3}}$ spaces as $ \varepsilon\to 0 $. The nonlinear problem is finally dealt with in Section \ref{section:HFsystem}, where, combining with the results of Section \ref{se:Strichartz}, we prove an existence result for the system \eqref{weakly_compressible_E}. Finally, we prove the lower bound estimate of the maximal lifespan $T^\star_\eps$ given in Theorem \ref{th:mainth}.

\section{Premilinaries} \label{se:dyadic}

The aim of this section is to briefly recall some elements of the Littlewood-Paley theory, which is the main technique used all along the paper. 

\subsection{Dyadic decomposition} We recall that in $\RR^d$, with $d\in\NN^*$, for $R > 0$, the ball $\mathcal{B}_d(0,R)$ is the set $$\mathcal{B}_d(0,R) = \set{\xi \in \RR^d \;:\; \abs{\xi} \leq R}.$$ For $0 < r_1 < r_2$, we defined the annulus
$$\mathcal{A}_d(r_1,r_2) \stackrel{\text{\tiny def}}{=} \set{\xi \in \RR^d \;:\; r_1 \leq \abs{\xi} \leq r_2}.$$ Next, we recall the following Bernstein-type lemma, which states that derivatives act almost as homotheties on distributions whose Fourier transforms are supported in a ball or an annulus. We refer the reader to \cite[Lemma 2.1.1]{C_book95} or \cite[Lemma 2.1]{BCDbook11} for a proof of this lemma.
\begin{lemma}
	\label{lemma:Bernstein}
	Let $k\in\NN$, $d \in \NN^*$ and $R, r_1, r_2 \in \RR$ satisfy $0 < r_1 < r_2$ and $R > 0$. There exists a constant $C > 0$ such that, for any $a, b \in \RR$, $1 \leq a \leq b \leq +\infty$, for any $\lambda > 0$ and for any $u \in L^a(\RR^d)$, we have
    \begin{equation}
        \label{eq:Bernstein1}
		\mbox{supp\,}\pare{\widehat{u}} \subset \mathcal{B}_d(0,\lambda R) \quad \Longrightarrow \quad \sup_{\abs{\alpha} = k} \norm{\dd^\alpha u}_{L^b} \leq C^k\lambda^{k+ d \pare{\frac{1}a-\frac{1}b}} \norm{u}_{L^a},
    \end{equation}
    \begin{equation}
        \label{eq:Bernstein2}
		\mbox{supp\,}\pare{\widehat{u}} \subset \mathcal{A}_d(\lambda r_1,\lambda r_2) \quad \Longrightarrow \quad C^{-k} \lambda^k\norm{u}_{L^a} \leq \sup_{\abs{\alpha} = k} \norm{\dd^\alpha u}_{L^a} \leq C^k \lambda^k\norm{u}_{L^a}.
    \end{equation}
\end{lemma}

In order to define the dyadic partition of unity, we also recall the following proposition, the proof of which can be found in \cite[Proposition 2.1.1]{C_book95} or \cite[Proposition 2.10]{BCDbook11}. 
\begin{prop}
	\label{pr:dyadic} Let $d \in \NN^*$. There exist smooth radial function $\chi$ and $\varphi$ from $\RR^d$ to $[0,1]$, such that
	\begin{gather}
		\label{eq:dyadic01} \mbox{supp\,}\chi \in \mathcal{B}_d\pare{0,\frac{4}3}, \quad \mbox{supp\,}\varphi \in \mathcal{A}_d\pare{\frac{3}4,\frac{8}3},\\
		\label{eq:dyadic02} \forall\, \xi \in \RR^3, \quad \chi(\xi) + \sum_{j\geqslant 0} \varphi(2^{-j}\xi) = 1,\\
		\label{eq:dyadic03} \abs{j-j'} \geqslant 2 \quad \Longrightarrow \quad \mbox{supp\,}\varphi(2^{-j}\cdot) \cap \mbox{supp\,}\varphi(2^{-j'}\cdot) = \varnothing,\\
		\label{eq:dyadic04} j \geqslant 1 \quad \Longrightarrow \quad \mbox{supp\,}\chi \cap \mbox{supp\,}\varphi(2^{-j}\cdot) = \varnothing.
	\end{gather}
	Moreover, for any $\xi \in \RR^d$, we have
	\begin{equation}
		\label{eq:dyadic05} \frac{1}2 \leqslant \chi^2(\xi) + \sum_{j\geqslant 0} \varphi^2(2^{-j}\xi) \leqslant 1.
	\end{equation}
\end{prop}
The dyadic blocks are defined as follows
\begin{defn}
	\label{de:dyadic} For any $d\in\NN^*$ and for any tempered distribution $u \in \mathcal{S}'(\RR^d)$, we set
    \begin{align*}
		&\tq u = \mathcal{F}^{-1} \pare{\varphi(2^{-q}\abs{\xi}) \widehat{u}(\xi)}, &&\forall q \in \NN,\\
		&\Delta_{-1} u = \mathcal{F}^{-1} \pare{\psi(\abs{\xi}) \widehat{u}(\xi)},&&\\
		&\tq u = 0, &&\forall q \leq -2,\\
		&S_q u = \sum_{q' \leq q - 1} \Delta_{q'} u, &&\forall q\geq 1,
    \end{align*}
    where $\mathcal{F}$ and $\mathcal{F}^{-1}$ stand for the Fourier transform and the inverse Fourier transform respectively.
\end{defn}
\noindent Using the properties of $\psi$ and $\varphi$, for any tempered distribution $u \in \mathcal{S}'(\RR^d)$, one can write
$$u = \sum_{q\geq -1} \tq u \qquad\mbox{in},$$ and the non-homogeneous Sobolev spaces $H^s(\RR^d)$, with $s\in\RR$, can be characterized as follows
\begin{prop}
    \label{pr:Sobnormiso}
    Let $d\in\NN^*$, $s\in\RR$ and $u\in H^s(\RR^d)$. Then,
    \begin{equation*}
        \norm{u}_{H^s} := \pare{\int_{\RR^d} (1 + \abs{\xi}^2)^s \abs{\widehat{u}(\xi)}^2 d\xi}^{\frac{1}2} \sim \pare{\sum_{q\geq -1} 2^{2qs} \norm{\tq u}_{L^2}^2}^{\frac{1}2}
    \end{equation*}
    Moreover, there exists a square-summable sequence of positive numbers $\{c_q(u)\}_q$ with $\sum_q c_q(u)^2 = 1$, such that 
	\begin{equation}
		\label{eq:DeltaqHs} \norm{\tq u}_{L^2} \leq c_q(u) 2^{-qs} \norm{u}_{H^s}.
	\end{equation}
\end{prop}

\subsection{Paradifferential calculus.} \label{paradiffcalcul}

The decomposition into dyadic blocks allows, at least formally, to write, for any tempered distributions $u$ and $v$,
\begin{align}
	uv = & \sum_{\substack{q\in\mathbb{Z} \\ q'\in\mathbb{Z}}}\tq u\, \Delta_{q'} v
\end{align}
The Bony decomposition (see for instance \cite{BCDbook11}, \cite{Bony81} or \cite{C_book95} for more details) consists in splitting the above sum in three parts. The first corresponds to the low frequencies of $u$ multiplied by the high frequencies of $v$, the second is the symmetric counterpart of the first, and the third part concerns the indices $q$ and $q'$ which are comparable. Then,
\begin{equation*}
	uv = T_u v+ T_v u + R\left(u,v\right),
\end{equation*}
where
\begin{align*}
	T_u v=& \sum_q S_{q-1} u \tq v\\
	R\left( u,v \right) = & \sum_{\abs{q-q'} \leqslant 1} \Delta_q u \Delta_{q'} v.
\end{align*}
Using the quasi-orthogonality given in \eqref{eq:dyadic03} and \eqref{eq:dyadic04}, we get the following relations.
\begin{lemma}
	\label{le:orthogonal}
	For any tempered distributions $u$ and $v$, we have
	\begin{align*}
		&\tq \pare{S_{q'-1} u \Delta_{q'} v} = 0 && \text{if } \left|q-q'\right|\geqslant 5\\
		&\tq \pare{S_{q'+1} u \Delta_{q'} v} = 0 && \text{if } q'\leqslant q-4.
	\end{align*}
\end{lemma}
\noindent Lemma \ref{le:orthogonal} implies the following decomposition, which we will widely use in this paper
\begin{equation}
	\label{eq:Bonydec} \tq (uv) = \spres \tq \pare{S_{q'-1} v \Delta_{q'} u} + \sloin \tq \pare{S_{q'+2} u \Delta_{q'} v}.
\end{equation}

Finally, we recall the following lemma, a proof of which can be found in \cite[p. 110]{BCDbook11}.

\begin{lemma}\label{commutator estimate}
	Let be $ p,q,r\in \left[ 1, \infty \right] $ such that $ \dfrac{1}{p}+\dfrac{1}{q}=\dfrac{1}{r} $ and $ f\in W^{1,p}\left( \R^3 \right) $, $ g\in L^q\left( \R^3 \right) $. Then
	\begin{equation}
		\left\| \left[ \tq , f\right] g \right\|_{L^r}\leqslant C 2^{-q} \left\| \nabla f \right\|_{L^p}\left\| g \right\|_{L^q},
	\end{equation}
	where the commutator $\left[\tq, f\right]g$ is defined as
	$$
	\left[\tq, f \right]g= \tq \left( fg \right) - f \tq g.
	$$
\end{lemma}

\subsection{Auxiliary estimates}

We first recall the following classical product rule in $\Hs$ spaces.
\begin{lemma}\label{le:LinfHsalgebra}
	For any $ s>0 $, there exists a constant $C$ such that, for any $u$, $v$ in $\Hs\cap\Linfty$, we have
	\begin{equation}
		\label{product_rule}
		\norm{uv}_{H^s} \leqslant \frac{C^{s+1}}{s} \pare{\norm{u}_{L^\infty} \norm{v}_{H^s} +\norm{v}_{L^\infty} \norm{u}_{H^s}}.
	\end{equation}
\end{lemma}
\noindent To prove Lemma \ref{le:LinfHsalgebra}, it suffices to decompose the data $ uv $ using the decomposition \eqref{eq:Bonydec} and apply repeatedly H\"older inequality.

In Section \ref{section:HFsystem}, we need the following Chemin-Lerner spaces $\widetilde{L}^p\pare{[0,t],\Hs}$, with $p \geqslant 2$ and $t > 0$, which are defined as the closure of the set of smooth vector-fields under the norms
\begin{equation*}
    \norm{u}_{\widetilde{L}^p\pare{[0,t],H^s}} = \Big(\sum_{q} 2^{2qs}\norm{\Delta_q u}_{L^p([0,t],L^2)}^2\Big)^{\frac{1}2}.
\end{equation*}
From the above definition, it is easy to see that, for any $p \geqslant 2$, $\widetilde{L}^p\pare{[0,t],\Hs}$ is smoother than $L^p\pare{[0,t],\Hs}$. From the above definition, we can prove the following lemma which gives similar estimates as \eqref{eq:DeltaqHs}.
\begin{lemma}
    \label{le:DeltaqHsbis} Suppose that $u$ belongs to $\widetilde{L}^p\pare{[0,t],\Hs}$, with $s > 0$, then there exists a square-summable sequence of positive numbers $\set{c_q(u)}_{q\geqslant -1}$, with $\displaystyle \sum_q c_q(u)^2 = 1$, such that 
	$$\norm{\Delta_q u}_{L^p([0,t],L^2)} \leq c_q(u) 2^{-qs} \norm{u}_{\widetilde{L}^p\pare{[0,t],H^s}}.$$
\end{lemma}
\noindent For functions in $\widetilde{L}^p\pare{[0,t],\Hs}$, we can prove similar estimates as in \eqref{product_rule}. 
\begin{lemma}\label{le:tildeHsalgebra}
	Let $T > 0$. For any $s > 0$, there exists a constant $C(s)$ depending on $s$ such that, for any $u$, $v$ in $\widetilde{L}^\infty([0,T],\Hs) \cap L^\infty([0,T],\Linfty)$, we have
	\begin{equation*}
		\norm{uv}_{\widetilde{L}^\infty([0,T],H^s)} \leqslant C(s) \pare{\norm{u}_{L^\infty([0,T],L^\infty)} \norm{v}_{\widetilde{L}^\infty([0,T],H^s)} + \norm{u}_{\widetilde{L}^\infty([0,T],H^s)} \norm{v}_{L^\infty([0,T],L^\infty)}}.
	\end{equation*}
\end{lemma}

Finally, we recall the definition of the weak-$L^p$ spaces and a refined version of Young's inequality that we need in Section \ref{se:Strichartz} (see \cite[Theorem 1.5]{BCDbook11} for a proof, for instance).
\begin{defn}
	For $1 < p < \infty$ and for any measurable function $f: \R^d \to \R$, we define the space
	\begin{equation*}
		L^{p,\infty}(\R^d) \stackrel{\text{\tiny def}}{=} \set{f : \R^d \to \R \mbox{ measurable }\;:\; \norm{f}_{L^{p,\infty}} < +\infty},
	\end{equation*}
	where the quasinorm 
	\begin{equation*}
		\norm{f}_{L^{p,\infty}} \stackrel{\text{\tiny def}}{=} \sup_{\lambda > 0} \lambda \mu\pare{\set{x \in \R^d \;:\; \abs{f(x)} > \lambda}}^{\frac{1}p},
	\end{equation*}
	and where $\mu$ is the usual Lebesgue measure on $\R^d$. 
\end{defn}
\begin{theorem}
	\label{th:HLS} Let $p,q,r \in ]1,\infty[$ satisfying 
	\begin{equation*}
		\frac{1}p + \frac{1}q = 1 + \frac{1}r.
	\end{equation*}
	Then, a constant $C > 0$ exists such that, for any $f \in L^{p,\infty}(\R^d)$ and $g \in L^q(\R^d)$, the convolution product $f \ast g$ belongs to $L^r(\R^d)$ and we have
	\begin{equation}
		\label{eq:HLS} \norm{f\ast g}_{L^r} \leqslant C \norm{f}_{L^{p,\infty}} \norm{g}_{L^q}.
	\end{equation}
\end{theorem}

\section{Decomposition of the initial data} \label{se:Data}

We recall that, for $0 < r < R$, in \eqref{eq:CrR}, we defined
\begin{equation*}
	\Crr= \set{\xi\in\R^3_\xi \;\big\vert\; \abs{\xi} \leqslant R, \abs{\xi_h} \geqslant r, \abs{\xi_3} \geqslant r}. 
\end{equation*}
Let $\psi$ a $\mathcal{C}^{\infty}$-function from $\RR^3$ to $\RR$ such that
\begin{equation*}
	\psi(\xi) = \left\{ \begin{aligned} &1 \qquad\mbox{if } \quad 0 \leqslant \abs{\xi} \leqslant 1\\&0 \qquad\mbox{if} \quad \abs{\xi} \geqslant 2 \end{aligned} \right.
\end{equation*}
and $\Psi_{r,R}: \RR^3 \to \RR$ the following frequency cut-off function
\begin{equation}
	\label{eq:PsirR} \Psi_{r,R}(\xi) = \psi \pare{\frac{\abs{\xi}}{R}} \pint{1- \psi \pare{\frac{\abs{\xi_h}}{r}} } \pint{ 1- \psi \pare{\frac{\abs{\xi_3}}{r}} }.
\end{equation}
Then, we have $\Psi_{r,R} \in \mathcal{D}(\RR^3)$, $\mbox{supp\,} \Psi_{r,R} \subset \mathcal{C}_{\frac{r}2,2R}$ and $\Psi_{r,R} \equiv 1 \mbox{ on }\mathcal{C}_{r,R}$. We will decompose $ U_0 $ in the following way
$$
U_0= \overline{U}_0 + \widetilde{U}_0,
$$
where   
\begin{equation*}
	\overline{U}_0 = \Prr U_0 = \Psi_{r,R}(D) U_0 = \mathcal{F}^{-1} \pare{\Psi_{r,R}(\xi) \widehat{U_0}(\xi)}.
\end{equation*}

Our goal is to get precise controls of the $\Hs$-norms of $\widetilde{U}_0$ with respect to the frequency cut-off radii $r$ and $R$. 
\begin{lemma}
	\label{le:smalll_data_Utilde}
	Let $s \geqslant 0$, $s_0 >0, \; p\in ]1,2]$ and the initial data $U_0 \in Y_{s,s_0,p}$, where $Y_{s,s_0,p}$ is defined as in \eqref{eq:DefY} and \eqref{eq:normY}. There exists $\delta > 0$ such that, for $R > 0$ large enough and $r = R^{-\delta}$, 
	$$
	\left\| \widetilde{U}_0 \right\|_{H^s} \leqslant C \;\mathcal{C}(U_0) R^{-s_0},
	$$
	where $\mathcal{C}(U_0)$ is defined in \eqref{eq:C0}.
\end{lemma}

\begin{proof}
By the definition of $\Hs$-norm, we have
\begin{align*}
	\left\| \widetilde{U}_0 \right\|_{H^s}^2 &\leqslant \int_{\substack{ \left| \xi_3 \right|<r\\ \left| \xi_h \right|<R}} \left( 1+\left| \xi \right|^2 \right)^s \left| \widehat{U}_0 \left( \xi \right) \right|^2 d\xi_3 d\xi_h + \int_{\substack{ \left| \xi_3 \right|<R\\ \left| \xi_h \right|<r}} \left( 1+\left| \xi \right|^2 \right)^s \left| \widehat{U}_0 \left( \xi \right) \right|^2 d\xi_3 d\xi_h\\
	& \quad + \int_{\left| \xi \right|>R}\left( 1+\left| \xi \right|^2  \right)^s \left| \widehat{U}_0 \left( \xi \right) \right|^2 d\xi\\
	&= I_1 + I_2 + I_3.
\end{align*}
In what follows, we denote as $ \mathcal{F}_h $ and  $ \mathcal{F}_v $ respectively the horizontal and vertical Fourier transforms. Let $q$, $p'$ be positive numbers such that $ q = \frac{p}{p-1} $, $ q'= \frac{q}{2} $ and $ p'=\frac{p}{2-p} $. Thus $ 1 < p \leqslant 2 \leqslant q $, and $ p'\in \left]1,\infty\right] $ and the following relations hold
\begin{equation*}
	\frac{1}{p}+ \frac{1}{q}= \frac{1}{p'}+ \frac{1}{q'}= 1.
\end{equation*}

For the first integral, we write
\begin{align*}
	I_1 &= \int_{\substack{ \left| \xi_3 \right|<r\\ \left| \xi_h \right|<R}} \pare{\frac{1+\left| \xi_h \right|^2 +  \xi_3^2}{1+\xi_3^2}}^s \pare{1+\xi_3^2}^s \abs{\widehat{U}_0 (\xi)}^2 d\xi_3 d\xi_h\\ 
	&\leqslant C R^{2s}\int_{\left| \xi_3 \right|<r}\int_{ \left| \xi_h \right|<R} \left( 1+  \xi_3^2 \right)^s \left| \widehat{U}_0 \left( \xi \right) \right|^2 d\xi_h d\xi_3 \\
	&\leqslant C R^{2s}\int_{\left| \xi_3 \right|<r}\int_{ \R^2_{\xi_h}} \left( 1+  \xi_3^2 \right)^s \left| \widehat{U}_0 \left( \xi \right) \right|^2 d\xi_h d\xi_3 .
\end{align*}
Plancherel theorem in the horizontal variable yields 
\begin{align*}
	I_1 &\leqslant C R^{2s}\int_{\left| \xi_3 \right|<r}\int_{ \R^2_{\xi_h}} \left( 1+  \xi_3^2 \right)^s \left| \widehat{U}_0 \left( \xi \right) \right|^2 d\xi_h d\xi_3\\ 
	&= C R^{2s}\int_{\left| \xi_3 \right|<r}\int_{ \R^2_{x_h}} \left( 1+  \xi_3^2 \right)^s \left|  \mathcal{F}_v{U}_0 \left(x_h, \xi_3 \right)   \right|^2 dx_h d\xi_3.
\end{align*}
Applying Fubini theorem and H\"older inequality in the vertical direction, we get 
\begin{align*}
	 I_1 &\leqslant C R^{2s} \pare{\int_{\left| \xi_3 \right|<r} \pare{1+  \xi_3^2}^{p's} }^\frac{1}{p'} \int_{\R^2_{x_h}} \pare{\int_{\R_{\xi_3}}  \abs{\mathcal{F}_v{U}_0 (x_h,\xi_3)}^{2q'} d\xi_3}^\frac{1}{q'} dx_h\\
	 &\leqslant C R^{2s}  r^\frac{1}{p'}  \int_{\R^2_{x_h}} \left( \int_{\R_{\xi_3}}  \left|  \mathcal{F}_v{U}_0 \left(x_h, \xi_3 \right)   \right|^{q} d\xi_3 \right)^\frac{2}{q} dx_h,
\end{align*}
Finally, we use Hausdorff-Young inequality in the vertical direction, taking into account the relation $r\sim  R^{-\delta}$, to obtain 
\begin{equation}
	\label{eq:bound_low_freq_horizontal}
	I_1 = \int_{\substack{ \left| \xi_3 \right|<r\\ \left| \xi_h \right|<R}} \left( 1+\left| \xi_h \right|^2 +  \xi_3^2 \right)^s \left| \widehat{U}_0 \left( \xi \right) \right|^2 d\xi_3 d\xi_h \leqslant CR^{2s-\frac{\delta}{p'}}\left\| U_0 \right\|^2_{L^2_hL^p_v}.
\end{equation}
Similar calculations lead to the following estimate for the second integral 
\begin{equation}
	\label{eq:bound_low_freq_vertical}
	I_2 = \int_{\substack{ \left| \xi_3 \right|<R\\ \left| \xi_h \right|<r}} \left( 1+\left| \xi_h \right|^2 +  \xi_3^2 \right)^s \left| \widehat{U}_0 \left( \xi \right) \right|^2 d\xi_3 d\xi_h \leqslant C R^{2s-\frac{\delta}{p'}}\left\| U_0 \right\|^2_{L^2_vL^p_h}.
\end{equation}
The third term contains only the very high frequencies, hence is much simpler to control 
\begin{equation}
\label{eq:bound_hi_freq}
	I_3 = \int_{\abs{\xi} > R} \pare{1+ \abs{\xi}^2}^{-s_0} \pare{1+ \abs{\xi}^2}^{s+s_0} \abs{\widehat{U}_0(\xi)}^2 d\xi \leqslant R^{-2s_0} \left\| U_0 \right\|_{H^{s+s_0}}^2.
\end{equation}

We choose the free parameter $\delta$ such that $$\frac{\delta}{p'}= 2 (s + s_0).$$ Combining the estimates \eqref{eq:bound_low_freq_horizontal} to \eqref{eq:bound_hi_freq}, we can conclude the proof.
\end{proof}

\section{Strichartz-type estimates for the linear system} \label{se:Strichartz}

We recall that the projector $\Prr$ associates any tempered distribution $f$ to
\begin{equation}
	\label{eq:PrR}
	\Prr f = \Psi_{r,R}(D) f = \mathcal{F}^{-1} \pare{\Psi_{r,R}(\xi) \widehat{f}(\xi)},
\end{equation}
where the function $\Psi_{r,R}$ is defined in \eqref{eq:PsirR}. In this section, we consider the following frequency cut-off free-wave system
\begin{equation}
	\label{eq:free_wave}
	\left\lbrace
	\begin{aligned}
		& \partial_t \bU = \frac{1}{\varepsilon} \B \bU\\
		& \left. \bU \right|_{t=0}= \Prr U_0.
	\end{aligned}
	\right.
\end{equation}
where the linear hyperbolic operator $ \mathcal{B} $ is defined in \eqref{definizione_B}. Since the system \eqref{eq:free_wave} is linear and the Fourier transform of the initial data are supported in $\mathcal{C}_{\frac{r}2,2R}$, the Fourier transform of the solution $ \bU \left( t \right) $ is also supported in $\mathcal{C}_{\frac{r}2,2R}$ for any $ t>0 $. The aim of the present section is to analyze the dispersive properties of system \eqref{eq:free_wave} as $ \varepsilon \to 0 $, i..e to prove the following theorem

\begin{theorem}
	\label{th:Strichartz}
	Let $q \in \pint{2,+\infty}$ and $p \geqslant \frac{4q}{q-2}$. For any $U_0\in\2 $, the system \eqref{eq:free_wave} has a global solution $\bU$ such that, 
	\begin{equation}
	\label{eq:Strichartz}
	\norm{\bU}_{L^p\pare{\R_+; L^q\pare{\R^3}}} \leqslant C R^{\frac{3}{2}-\frac{3}{q}} r^{-\frac{4}{p}} \varepsilon^{\frac{1}{p}} \norm{U_0}_{\2}.
	\end{equation}
\end{theorem}

\bigskip

Writing the system \eqref{eq:free_wave} in Fourier frequency variable, we get
\begin{equation}
	\label{eq:FWFourier}
	\left\lbrace
	\begin{aligned}
		& \partial_t \widehat{\bU} = \frac{1}{\varepsilon} \widehat{\B}\, \widehat{\bU}\\
		& \widehat{\bU}|_{t=0}= \Psi_{r,R}(\xi) \, \widehat{U_0},
	\end{aligned}
	\right.
\end{equation}
where 
\begin{equation*}
	\widehat{\B} (\xi) = 	\begin{pmatrix}
		0&1&0& -i \bgam\xi_1\\
		-1 & 0 & 0 & -i \bgam\xi_2\\
		0&0&0 & -i \bgam\xi_3\\
		-i\bgam\xi_1 & -i \bgam\xi_2 & -i \bgam\xi_3 & 0
	\end{pmatrix}.
\end{equation*}
The characteristic polynomial of $\widehat{\B} (\xi)$ writes
\begin{equation}
	\label{char_polyn}
	\mathcal{P}_{\widehat{\B} (\xi)} \left( \lambda \right)= \det \left( \widehat{\B} (\xi) - \lambda \mathbb{I}_{\R^4} \right)= \lambda^4 + \pare{1 + \bgam^2 \abs{\xi}^2} \lambda^2 + \bgam^2 \xi^2_3.
\end{equation}
So, straightforward calculations shows that the eigenvalues of $\widehat{\B} (\xi)$ are 
\begin{equation*}
	\lambda_{\epsilon_1,\epsilon_2}(\xi) =  \epsilon_1 i \sqrt{\frac{1}{2} \pare{ \pare{1+ \bgam^2 \abs{\xi}^2} + \epsilon_2 \sqrt{\pare{1 + \bgam^2 \abs{\xi}^2}^2 - 4 \bgam^2 \xi^2_3} }}.
\end{equation*}
where $ \epsilon_1,\epsilon_2 \in \set{-1,1}$. We recall that, for any $A, B \in \RR$, $A^2 \geqslant B$, we have
\begin{equation*}
	\sqrt{A\pm \sqrt{B}}= \sqrt{\frac{A+\sqrt{A^2-B}}{2}}\pm \sqrt{\frac{A-\sqrt{A^2-B}}{2}}.
\end{equation*}
Then, setting 
\begin{align*}
	A &= 1 + \bgam^2 \abs{\xi}^2\\
	B &= \pare{1 + \bgam^2 \abs{\xi}^2}^2 - 4 \bgam^2 \xi_3^2,
\end{align*}
we can rewrite the eigenvalues as
\begin{equation}
	\label{eigenvalues}
	\lee (\xi) = \epsilon_1 \frac{i}{2} \pare{\sqrt{1 + \bgam^2 \abs{\xi}^2 + 2\bgam \xi_3} + \epsilon_2 \sqrt{1 + \bgam^2 \abs{\xi}^2 -2 \bgam \xi_3}}.
\end{equation}
We remark that a similar spectral analysis has already been performed in the work \cite{FGN} with the difference that the domain considered in \cite{FGN} was of the form $ \R^2_h \times \T^1_v $ instead that $ \R^3 $.

\bigskip

Now, in order to understand the behavior of the solutions to \eqref{eq:free_wave} we define the following operators
\begin{align}
	G_{\lambda} \left( t \right) f\left( x \right) &=  \mathcal{F}^{-1} \left( e^{\frac{t}{\varepsilon}\lambda \pare{\xi}} \widehat{f} \pare{\xi} \right) (x) = \int_{\R^3_\xi\times \R^3_y} {f} \pare{y} e^{\frac{t}{\varepsilon} \lambda\pare{\xi} + i \pare{x-y} \cdot \xi} d \xi d y, \notag
\end{align}
where the eigenvalues $\lambda\pare{\xi}$ are given in \eqref{eigenvalues}
\begin{equation*}
	\lambda\pare{\xi} =  \pm \frac{i}{2} \pare{ \sqrt{1 + \bgam^2 \abs{\xi}^2 + 2 \bgam \xi_3 } + \pm \sqrt{1 + \bgam^2 \abs{\xi}^2 - 2 \bgam \xi_3} }.
\end{equation*}

\begin{lemma}
	\label{le:Strichartz01}
	Let $0 < r \ll 1 \ll R$ and recall that
	\begin{equation*}
		\mathcal{C}_{r,R} = \set{\xi \in \R^3 \;\big\vert\; \abs{\xi_h}, \abs{\xi_3} \geqslant r, \; \abs{\xi} \leqslant R}.
	\end{equation*}
	Let $\Psi_{r,R}: \R^3 \to \R$ be a smooth function such that $\text{supp } \Psi_{r,R} \subset \mathcal{C}_{\frac{r}{2},2R}$ and ${\Psi_{r,R}}_{\vert_{\mathcal{C}_{r,R}}} \equiv 1$. Then,
	\begin{equation*}
		\norm{\Psi_{r,R}(D) G_\lambda(t) f}_{L^\infty\pare{\R^3}} \leqslant C R^3 r^{-2} \pare{\frac{\eps}{t}}^{\frac{1}{2}} \norm{f}_{L^1\pare{\R^3}}.
	\end{equation*}
\end{lemma}

\bigskip

The standard method to prove Lemma \ref{le:Strichartz01} is to follow the microlocal study of $G_\lambda(t) f$ on each dyadic ring in Fourier variable as in \cite{CDGG}. Here, in order to have a self-contained text which is short enough, we present the complete proof of Lemma \ref{le:Strichartz01}, using a simplified method as in \cite{VSN}. We write
\begin{align*}
	\Psi_{r,R}(D) G_{\lambda} \left( t \right) f(x) = \int_{\R^3_y} f(y) \int_{\R^3_\xi} e^{\frac{t}{\eps} \lambda(\xi) + i(x-y)\cdot \xi} \Psi_{r,R}(\xi) d\xi d y= K_\lambda \pare{\frac{t}{\eps},\cdot} \ast f (x),
\end{align*}
where
\begin{equation}
	\label{eq:Kljk} 
	K_\lambda \pare{\tau,x} = \int_{\R^3_\xi} e^{\tau \lambda(\xi) + ix\cdot \xi} \Psi_{r,R}(\xi) d\xi.
\end{equation}
We remark that the invariance of $K_\lambda$ by rotation in the plane $\R^2_{\xi_h}$ allows to restrict the study to the case $x_2 = 0$. Indeed, if $x_2 \neq 0$, we can perform a rotation of angle $\theta$, with $\cot \theta = \frac{x_1}{x_2}$ to suppress the second component of $x$. Let
\begin{align*}
	A(\xi) &= \sqrt{1 + \bgam^2 \abs{\xi}^2 + 2\bgam\xi_3}\\
	B(\xi) &= \sqrt{1 + \bgam^2 \abs{\xi}^2 - 2\bgam\xi_3}\\
	\lambda(\xi) &= \pm \frac{i}{2} \pare{A(\xi) \pm B(\xi)}\\
	a_\lambda(\xi) &= \dd_{\xi_2} \lambda(\xi) = \pm \frac{i \bgam^2 \xi_2}{2} \pare{\frac{1}{A(\xi)} \pm \frac{1}{B(\xi)}}.
\end{align*}
We consider the operator
\begin{equation}
	\label{eq:Llambda}
	\mathcal{L}_\lambda \stackrel{\text{\tiny def}}{=} \frac{1}{1 + \tau a_\lambda^2} \pare{\text{Id} - ia_\lambda \dd_{\xi_2}}.
\end{equation}
Direct calculations give
\begin{equation*}
	\mathcal{L}_\lambda \pare{e^{\tau \lambda(\xi) + ix\cdot\xi}} = e^{\tau \lambda(\xi) + ix\cdot\xi},
\end{equation*}
thus, an integration by parts gives
\begin{equation*}
	K_\lambda(\tau,x) = \int_{\R^3_\xi} e^{\tau \lambda(\xi) + ix\cdot \xi} \, {}^T\!\!\mathcal{L}_\lambda \pare{\Psi_{r,R}(\xi)} d\xi,
\end{equation*}
where
\begin{equation}
	\label{eq:TLlambda}
	{}^T\!\!\mathcal{L}_\lambda \pare{\Psi_{r,R}(\xi)} = \pint{\frac{1}{1 + \tau a_\lambda^2} + i\pare{\dd_{\xi_2} a_\lambda} \frac{1 - \tau a_\lambda^2}{\pare{1 + \tau a_\lambda^2}^2}} \Psi_{r,R}(\xi) + \frac{ia_\lambda}{1 + \tau a_\lambda^2} \dd_{\xi_2} \Psi_{r,R}(\xi).
\end{equation}

\bigskip

\begin{lemma}
	\label{le:TLlambda}
	Let $0 < r \ll 1 \ll R$. There exists a constant $C > 0$ such that
	\begin{equation*}
		\abs{{}^T\!\!\mathcal{L}_\lambda \pare{\Psi_{r,R}(\xi)}} \leqslant \frac{C r^{-1}}{1 + R^{-6} r^2 \tau \xi_2^2}.
	\end{equation*}
\end{lemma}

\begin{proof}

Since $\text{supp } \Psi_{r,R} \subset \mathcal{C}_{\frac{r}{2},2R}$, there exist constants $C_1, C_2 > 0$ such that
\begin{align*}
	&C_1 r \leqslant A(\xi) = \sqrt{\bgam^2 \abs{\xi_h}^2 + \pare{1 + \bgam \xi_3}^2} \leqslant C_2 R\\
	&C_1 r \leqslant B(\xi) = \sqrt{\bgam^2 \abs{\xi_h}^2 + \pare{1 - \bgam \xi_3}^2} \leqslant C_2 R.
\end{align*}
Remark that
\begin{equation*}
	\abs{\frac{1}{A(\xi)} - \frac{1}{B(\xi)}} = \frac{\abs{A(\xi) - B(\xi)}}{A(\xi) + B(\xi)} = \frac{\abs{A(\xi)^2 - B(\xi)^2}}{A(\xi)B(\xi)\pare{A(\xi) + B(\xi)}} = \frac{4 \bgam \abs{\xi_3}}{A(\xi)B(\xi)\pare{A(\xi) + B(\xi)}}.
\end{equation*}
Then, we can choose $C_1$ and $C_2$ such that
\begin{equation}
	\label{eq:ajzeta}
	C_1 R^{-3} r \abs{\xi_2} \leqslant \abs{a_\lambda(\xi)} \leqslant C_2 r^{-1}.
\end{equation}
Differentiating $a_\lambda$ with respect to $\xi_2$, we get
\begin{align*}
	\dd_{\xi_2} a_\lambda(\xi) &= \pm \frac{i\bgam^2}2 \pare{\frac{1}{A(\xi)} \pm \frac{1}{B(\xi)}} - \frac{i\bgam^4\xi_2^2}2 \pare{\frac{1}{A(\xi)^3} \pm \frac{1}{B(\xi)^3}}.
\end{align*}
Since $A(\xi), B(\xi) \geq \bgam \abs{\xi_2}$, by adjusting $C_2$, we can say that
\begin{equation}
	\label{eq:d2ajzeta}
	\abs{\dd_{\xi_2} a_\lambda(\xi)} \leqslant C_2 r^{-1}.
\end{equation}

\bigskip

So, Estimates \eqref{eq:ajzeta} and \eqref{eq:d2ajzeta} yield
\begin{gather*}
	\abs{\frac{\Psi_{r,R}(\xi)}{1 + \tau a_\lambda^2}} \leqslant \frac{C}{1 + R^{-6} r^2 \tau \xi_2^2}\\
	\abs{i\pare{\dd_{\xi_2} a_\lambda} \frac{1 - \tau a_\lambda^2}{\pare{1 + \tau a_\lambda^2}^2} \Psi_{r,R}(\xi)} \leqslant \frac{C r^{-1}}{1 + R^{-6} r^2 \tau \xi_2^2}\\
	\abs{\frac{ia_\lambda}{1 + \tau a_\lambda^2} \dd_{\xi_2} \Psi_{r,R}(\xi)} \leqslant \frac{C r^{-1}}{1 + R^{-6} r^2 \tau \xi_2^2},
\end{gather*}
which imply
\begin{equation*}
	\abs{{}^T\!\!\mathcal{L}_\lambda \pare{\Psi_{r,R}(\xi)}} \leqslant \frac{C r^{-1}}{1 + R^{-6} r^2 \tau \xi_2^2}.
\end{equation*}

\end{proof}

\begin{lemma}
	\label{eq:tildeKernel}
	For any $\tau > 0$,  
	\begin{equation*}
		\norm{K_\lambda(\tau,\cdot)}_{L^\infty_x} \leqslant C R^3 r^{-2} \tau^{-\frac{1}{2}}.
	\end{equation*}
\end{lemma}

\begin{proof}
	We recall that 
	\begin{equation*}
		K_\lambda(\tau,x) = \int_{\R^3_\xi} e^{\tau \lambda(\xi) + ix\cdot \xi} \, {}^T\!\!\mathcal{L}_\lambda \pare{\Psi_{r,R}(\xi)} d\xi,
	\end{equation*}
	Then, using Lemma \ref{le:TLlambda} and a change of variable with respect to $\xi_2$, we obtain
	\begin{align*}
		\norm{K_\lambda(\tau,\cdot)}_{L^\infty_x} \leqslant C r^{-1} \int_{\RR} \frac{d \xi_2}{1 + R^{-6} r^2 \tau \xi_2^2} \leqslant C R^3 r^{-2} \tau^{-\frac{1}{2}} \int_{\RR} \frac{d \zeta}{1 + \zeta^2} \leqslant C R^3 r^{-2} \tau^{-\frac{1}{2}},
	\end{align*}
	where $C$ is a generic positive constant that can be ajusted from line to line.
\end{proof}

\noindent \textit{Proof of Lemma \ref{le:Strichartz01}.} We recall that
\begin{equation*}
	\Psi_{r,R}(D) G_{\lambda} \left( t \right) f(x) = K_\lambda \pare{\frac{t}{\eps},\cdot} \ast f (x).
\end{equation*}
Using Young's inequality, we obtain
\begin{equation*}
	\norm{\Psi_{r,R}(D) G_{\lambda} \left( t \right) f}_{L^\infty\pare{\R^3}} \leqslant C \norm{K_\lambda\pare{\frac{t}{\eps},\cdot}}_{L^\infty_x} \norm{f}_{L^1_x}
	\leqslant C R^3 r^{-2} \pare{\frac{\eps}{t}}^{\frac{1}{2}} \norm{f}_{L^1\pare{\R^3}}.
\end{equation*}
$\hfill \square$

\begin{rem}
	\label{re:Strichartz02}
	It is clear that the estimates in Lemma \ref{le:Strichartz01} are not optimal for $t \leqslant \eps$. Indeed, for $t \leqslant \eps$, since $\text{supp } \Psi_{r,R} \subset \mathcal{C}_{\frac{r}{2},2R} \subset \mathcal{B}_3(0,2R)$, using Bernstein lemma \ref{lemma:Bernstein}, we can simply bound
	\begin{align}
		\label{eq:lemma1_smallt}
		\norm{\Psi_{r,R}(D) G_\lambda(t) f}_{L^\infty\pare{\R^3}} \leqslant C R^3 \norm{f}_{L^1\pare{\R^3}}.
	\end{align}
\end{rem}

\bigskip

\begin{lemma}
	\label{le:Strichartz03}
	For any $q \in [2,+\infty]$ and $\overline{q} \in [1,2]$ such that $\frac{1}{q} + \frac{1}{\overline{q}} = 1$, we have
	\begin{equation*}
		\norm{\Psi_{r,R}(D) G_\lambda(t) f}_{L^q\pare{\R^3}} \leqslant C \pint{R^3 \min \set{1, r^{-2} \pare{\frac{\eps}{t}}^{\frac{1}{2}}}}^{1-\frac{2}{q}} \norm{f}_{L^{\overline{q}}\pare{\R^3}}.
	\end{equation*}
\end{lemma}

\begin{proof}
	We already prove that
	\begin{equation*}
		\norm{\Psi_{r,R}(D) G_\lambda(t) f}_{L^\infty\pare{\R^3}} \leqslant C R^3 \min \set{1, r^{-2} \pare{\frac{\eps}{t}}^{\frac{1}{2}}} \norm{f}_{L^1\pare{\R^3}}
	\end{equation*}
	The definition of $\Psi_{r,R}(D) G_\lambda(t)$ implies that
	\begin{equation*}
		\norm{\Psi_{r,R}(D) G_\lambda(t) f}_{L^2\pare{\R^3}} \leqslant C \norm{f}_{L^2\pare{\R^3}}.
	\end{equation*}
	Since the point $\pare{\frac{1}{q},\frac{1}{\overline{q}}}$ belongs to the line segment $\pint{\pare{0,1},\pare{\frac{1}{2},\frac{1}{2}}}$, the Riesz-Thorin theorem yields
	\begin{equation*}
		\norm{\Psi_{r,R}(D) G_\lambda(t) f}_{L^q\pare{\R^3}} \leqslant C \pint{R^3 \min \set{1, r^{-2} \pare{\frac{\eps}{t}}^{\frac{1}{2}}}}^{1-\frac{2}{q}} \norm{f}_{L^{\overline{q}}\pare{\R^3}}.
	\end{equation*} 
\end{proof}

The following theorem gives Strichartz estimates of $\bU$ in the direction of each eigenvector of the operator $\widehat{\mathcal{B}}$.

\begin{theorem}
	\label{th:Strichartz04}
	Let $q \in [2,+\infty]$ and $p \geqslant \frac{4q}{q-2}$. Then,
	\begin{equation*}
		\norm{\Psi_{r,R}(D) G_\lambda(t) f}_{L^p_t\pare{L^q\pare{\R^3}}} \leqslant C R^{\frac{3}{2} - \frac{3}{q}} r^{-\frac{4}{p}} \eps^{\frac{1}{p}} \norm{f}_{L^2\pare{\R^3}}.
	\end{equation*}
\end{theorem}

\begin{proof}
	Following the ideas of \cite{CDGG} and \cite{CDGG2}, we will apply the so-called $TT^*$ method, which consist in an argument of duality. Let $\overline{p}$ and $\overline{q}$ such that 
	\begin{equation*}
		\frac{1}{p} + \frac{1}{\overline{p}} = \frac{1}{q} + \frac{1}{\overline{q}} = 1,
	\end{equation*}
	and
	\begin{equation*}
		\Xi = \set{\varphi \in \mathcal{D}\pare{\mathbb{R}_+ \times \R^3} \;\vert\; \norm{\varphi}_{L^{\overline{p}}_t\pare{L^{\overline{q}}_x}} \leqslant 1}.
	\end{equation*}
	Considering $\Phi = \Psi_{r,R}(D) \varphi$ and using Plancherel theorem and Cauchy-Schwarz inequality, we have
	\begin{align*}
		\norm{\Psi_{r,R}(D) G_\lambda(t) f}_{L^p_t\pare{L^q\pare{\R^3}}} &= \sup_{\varphi \in \Xi} \int_{\RR_+} \psca{\Psi_{r,R}(D) G_\lambda(t) f \;,\; \varphi}_{L^2_x} d t\\
		&= \pare{2\pi}^{-3} \sup_{\varphi \in \Xi} \int_{\RR_+ \times \RR^3_\xi} \widehat{f}\pare{\xi} \widehat{\Phi}\pare{t,\xi} e^{\frac{t}{\eps} \lambda(\xi)} dt\, d\xi\\
		&\leqslant \pare{2\pi}^{-3} \sup_{\varphi \in \Xi} \norm{f}_{L^2} \norm{\int_{\RR_+} \widehat{\Phi}\pare{t,\xi} e^{\frac{t}{\eps} \lambda(\xi)} d t}_{L^2_\xi}.
	\end{align*}

	It remains to estimate
	\begin{equation*}
		I = \norm{\int_{\RR_+} \widehat{\Phi}\pare{t,\xi} e^{\frac{t}{\eps} \lambda(\xi)} d t}_{L^2_\xi}.
	\end{equation*}
	Recalling that $\lambda(\xi)$ is an imaginary number, using several times Fubini's theorem, Plancherel theorem and H\"older's inequality, we have
	\begin{align*}
		I^2 &= \psca{\int_{\RR_+} \widehat{\Phi}(t,\xi) e^{\frac{t}{\eps}\lambda(\xi)} d t, \int_{\RR_+} \overline{\widehat{\Phi}(s,\xi) e^{\frac{s}{\eps}\lambda(\xi)}} d s}_{L^2}\\
		&= \int_{\RR^3_\xi} \pare{\int_{\RR_+} \widehat{\Phi}(t,\xi) e^{\frac{t}{\eps} \lambda(\xi)} d t} \pare{\int_{\RR_+} \overline{\widehat{\Phi}(s,\xi)} e^{-\frac{s}{\eps} \lambda(\xi)} d s} d\xi\\
		&= \int_{\RR^3_\xi} \pare{\int_{\pare{\RR_+}^2} \widehat{\Phi}(t,\xi) \overline{\widehat{\Phi}(s,\xi)} e^{-\frac{t-s}{\eps} \lambda(\xi)} ds\, dt} d\xi\\
		&= \int_{\pare{\RR_+}^2} \int_{\RR^3_\xi} \big(\Psi_{r,R}(D) \varphi(s,-x)\big) \big(\Psi_{r,R}(D) G_\lambda(t-s) \varphi(t,x)\big) dx\, dt\, ds\\
		&\leqslant C \int_{\pare{\RR_+}^2} \norm{\varphi(s)}_{L^{\overline{q}}_x} \norm{\Psi_{r,R}(D) G_\lambda(t-s)\varphi(t)}_{L^q_x} dt\, ds. 
	\end{align*}
	Next, using Lemma \ref{le:Strichartz03}, H\"older's inequality, we get
	\begin{align*}
		I^2 &\leqslant C \int_{\pare{\RR_+}^2} \norm{\varphi(s)}_{L^{\overline{q}}_x} \norm{\varphi(t)}_{L^{\overline{q}}_x} \pint{R^3 \min \set{1, \frac{r^{-2} \eps^{-\frac{1}2}}{\abs{t-s}^{\frac{1}2}}}}^{1-\frac{2}q} ds\, dt\\
		&\leqslant C \norm{\varphi}_{L^{\overline{p}}_t \pare{L^{\overline{q}}_x}} \norm{\int_{\RR_+} \norm{\varphi(s)}_{L^{\overline{q}}_x} \pint{R^3 \min \set{1, \frac{r^{-2} \eps^{-\frac{1}2}}{\abs{t-s}^{\frac{1}2}}}}^{1-\frac{2}q} ds}_{L^p_t}\\
		&\leqslant C \norm{\varphi}_{L^{\overline{p}}_t \pare{L^{\overline{q}}_x}} R^{3\pare{1-\frac{2}{q}}} \norm{\norm{\varphi(\cdot)}_{L^{\overline{q}}_x} \ast_t M(\cdot)}_{L^p_t},
	\end{align*}
	where $$M(t) = \pint{\min \set{1, \frac{r^{-2} \eps^{-\frac{1}2}}{\abs{t}^{\frac{1}2}}}}^{1-\frac{2}q}.$$	
	
	If $(p,q)=(+\infty,2)$, Theorem \ref{th:Strichartz04} is obvious from the definition of $\Psi_{r,R}(D) G_\lambda(t)$. In the case where $q > 2$, we study two different cases
	\begin{itemize}
		\item If $p > \frac{4q}{q-2}$ then $M \in L^{\frac{p}2}_t$. For any $q > 2$,
		\begin{align*}
			\norm{M}_{L^{\frac{p}2}_t} &= \pare{\int_0^{r^{-4}\eps} dt + \int_{r^{-4}\eps}^{+\infty} \pare{\frac{r^{-4}\eps}t}^{\frac{p}2 \pare{\frac{1}2 - \frac{1}q}} dt}^{\frac{2}p}\\
			&= \pare{r^{-4}\eps}^{\frac{2}p} \pare{1 + \int_1^{+\infty} \pare{\frac{1}{\tau}}^{\frac{p}2 \pare{\frac{1}2 - \frac{1}q}} d\tau}^{\frac{2}p} \leqslant C \pare{r^{-4}\eps}^{\frac{2}p}.
		\end{align*}
		Thus, using the classical Young's inequality, we obtain
		\begin{equation}
		 	I^2 \leqslant C R^{3\pare{1-\frac{2}{q}}} \norm{\varphi}_{L^{\overline{p}}_t \pare{L^{\overline{q}}_x}}^2 \norm{M}_{L^{\frac{p}2}_t} \leqslant C R^{3\pare{1-\frac{2}{q}}} \pare{r^{-4}\eps}^{\frac{2}p} \norm{\varphi}_{L^{\overline{p}}_t \pare{L^{\overline{q}}_x}}^2.
		 \end{equation}	
	 
		 \item If $p = \frac{4q}{q-2}$ then Young's inequality does not work anymore because $M \notin L^{\frac{p}2}_t$. However, it is well known that $\abs{t}^{-\frac{2}{p}}$ belongs to $L^{\frac{p}2,\infty}\pare{\R}$, so $$M(t) = \pint{\min \set{1, \pare{r^{-4}\eps}^{\frac{2}p} \abs{t}^{-\frac{2}{p}}}} \in L^{\frac{p}2,\infty}\pare{\RR}$$ and 
		 $$\norm{M}_{L^{\frac{p}2,\infty}_t} \sim \norm{\abs{\cdot}^{-\frac{2}{p}}}_{L^{\frac{p}2,\infty}_t}\pare{r^{-4}\eps}^{\frac{2}p}.$$
		 Applying Theorem \ref{th:HLS}, we also get
		 \begin{equation}
			 \label{eq:where to apply HLS} I^2 \leqslant C R^{3\pare{1-\frac{2}{q}}} \norm{\varphi}_{L^{\overline{p}}_t \pare{L^{\overline{q}}_x}}^2 \norm{M}_{L^{\frac{p}2,\infty}_t} \leqslant C R^{3\pare{1-\frac{2}{q}}} \pare{r^{-4}\eps}^{\frac{2}p} \norm{\varphi}_{L^{\overline{p}}_t \pare{L^{\overline{q}}_x}}^2.
		 \end{equation}
	\end{itemize}
	
	We deduce that,
	\begin{equation*}
		I \leqslant C R^{\frac{3}2\pare{1-\frac{2}{q}}} \pare{r^{-4} \eps}^{\frac{1}p} \norm{\varphi}_{L^{\overline{p}}_t \pare{L^{\overline{q}}_x}},
	\end{equation*}
	and, 
	\begin{equation*}
		\norm{\Psi_{r,R}(D) G_\lambda(t) f}_{L^p_t\pare{L^q_x}} \leqslant C \pare{2\pi}^{-3} R^{\frac{3}{2} - \frac{3}{q}} r^{-\frac{4}{p}} \eps^{\frac{1}{p}} \norm{f}_{L^2_x}.
	\end{equation*}
\end{proof}

\begin{rem}
	We want to make some remarks about the dispersive result in Theorem \ref{th:Strichartz04}.
	\begin{enumerate}
		\item Unlike the case of viscous fluids (see for instance \cite{CDGG}, \cite{CDGG2} or \cite{VSN}), we cannot obtain dispersive estimates for $\Psi_{r,R}(D) G_\lambda(t) f$ in an $L^1_t(L^q_x)$-norm, due to the fact that we do not have damping effect given by the viscosity terms. This is one of the main reasons why we can only obtain a longtime existence result.

		\bigskip

		\item The result of Theorem \ref{th:Strichartz04} is slightly better than the dispersive estimates obtained in \cite{Dutrifoy2} in the sense that we can treat the limit case $p = \frac{4q}{q-2}$, using Theorem \ref{th:HLS} to get \eqref{eq:where to apply HLS}. In general cases, \eqref{eq:where to apply HLS} is known as the Hardy-Littlewood-Sobolev inequality, where one uses the fact that the function $\abs{x}^{-\frac{d}{p}}$ belongs to $L^{p,\infty}\pare{\R^d}$ but not to $L^p\pare{\R^d}$.
	\end{enumerate}
\end{rem}

\bigskip

\noindent \textit{Proof of Theorem \ref{th:Strichartz}.} We recall that in the Fourier variable, the system \eqref{eq:free_wave} writes
\begin{equation}
	\tag{\ref{eq:FWFourier}}
	\left\lbrace
	\begin{aligned}
		& \partial_t \widehat{\bU} = \frac{1}{\varepsilon} \widehat{\B}\, \widehat{\bU}\\
		& \widehat{\bU}|_{t=0} = \Psi_{r,R}(\xi) \widehat{U_0},
	\end{aligned}
	\right.
\end{equation}
where
\begin{equation*}
	\widehat{\B} = 	\begin{pmatrix}
					0&1&0& -i \bgam \xi_1\\
					-1 & 0 & 0 & -i \bgam\xi_2\\
					0&0&0 & -i \bgam\xi_3\\
					-i\bgam\xi_1 & -i \bgam\xi_2 & -i \bgam\xi_3 & 0
				\end{pmatrix}.
\end{equation*}
We also recall that the eigenvalues of $\widehat{\B}$ are
\begin{equation*}
	\lee (\xi) = \epsilon_1 \frac{i}{2} \pare{\sqrt{1+ \bgam^2\abs{\xi}^2 +2 \bgam \xi_3} + \epsilon_2 \sqrt{1+ \bgam^2\abs{\xi}^2 - 2 \bgam\xi_3}},
\end{equation*}
with $\epsilon_1, \epsilon_2 \in \set{-1, 1}$. Since the support of $\Psi_{r,R}$ is included in $\mathcal{C}_{\frac{r}2,2R}$ defined as in \eqref{eq:CrR}, we can suppose that $\xi_3 \neq 0$. So $\widehat{\B}$ admits four distinct eigenvalues and since $\widehat{\B}$ is a skew-Hermitian matrix, the unit eigenvectors $\overrightarrow{V}_{\epsilon_1, \epsilon_2}(\xi)$ corresponding to the eigenvalues $\lee (\xi)$ form an orthonormal basis of $\RR^4$. Decomposing $$\widehat{U_0}(\xi) = \sum_{\epsilon_1, \epsilon_2 \in \set{-1, 1}} C_{\epsilon_1, \epsilon_2}(\xi) \, \overrightarrow{V}_{\epsilon_1, \epsilon_2}(\xi),$$ the solution of the system \eqref{eq:FWFourier} writes
\begin{equation*}
	\widehat{\bU}(t,\xi) = \sum_{\epsilon_1, \epsilon_2 \in \set{-1, 1}} \Psi_{r,R}(\xi) e^{\frac{t}{\eps} \lee(\xi)} \, C_{\epsilon_1, \epsilon_2}(\xi) \, \overrightarrow{V}_{\epsilon_1, \epsilon_2}(\xi).
\end{equation*}

Before performing the following calculations, we recall that thoughout the paper, we always use $C$ to denote generic positive constants that can change from line to line. Applying Theorem \ref{th:Strichartz04} and Plancherel theorem, we have
\begin{align*}
	\norm{\bU}_{L^p\pare{\R_+; L^q\pare{\R^3}}} &\leqslant \sum_{\epsilon_1, \epsilon_2 \in \set{-1, 1}} \norm{\mathcal{F}^{-1} \pare{\Psi_{r,R}(\xi) e^{\frac{t}{\eps} \lee(\xi)} \, C_{\epsilon_1, \epsilon_2}(\xi) \, \overrightarrow{V}_{\epsilon_1, \epsilon_2}(\xi)}}_{L^p\pare{\R_+; L^q\pare{\R^3}}}\\
	&\leqslant C R^{\frac{3}{2} - \frac{3}{q}} r^{-\frac{4}{p}} \eps^{\frac{1}{p}} \sum_{\epsilon_1, \epsilon_2 \in \set{-1, 1}} \norm{C_{\epsilon_1, \epsilon_2}(\xi) \, \overrightarrow{V}_{\epsilon_1, \epsilon_2}(\xi)}_{L^2_\xi}.
\end{align*}
Using Cauchy-Schwarz inequality and the orthogonality of $\set{\overrightarrow{V}_{\epsilon_1, \epsilon_2}(\xi)}$, we get
\begin{align*}
	\sum_{\epsilon_1, \epsilon_2 \in \set{-1, 1}} \norm{C_{\epsilon_1, \epsilon_2}(\xi) \, \overrightarrow{V}_{\epsilon_1, \epsilon_2}(\xi)}_{L^2_\xi} &= \sum_{\epsilon_1, \epsilon_2 \in \set{-1, 1}} \pare{\int_{\RR^3} \abs{C_{\epsilon_1, \epsilon_2}(\xi) \, \overrightarrow{V}_{\epsilon_1, \epsilon_2}(\xi)}^2 d\xi}^{\frac{1}2}\\
	&\leqslant \pare{\sum_{\epsilon_1, \epsilon_2 \in \set{-1, 1}} 1}^{\frac{1}2} \pare{\int_{\RR^3} \sum_{\epsilon_1, \epsilon_2 \in \set{-1, 1}} \abs{C_{\epsilon_1, \epsilon_2}(\xi) \, \overrightarrow{V}_{\epsilon_1, \epsilon_2}(\xi)}^2 d\xi}^{\frac{1}2}\\
	&= 2 \pare{\int_{\RR^3} \Big\vert\sum_{\epsilon_1, \epsilon_2 \in \set{-1, 1}} C_{\epsilon_1, \epsilon_2}(\xi) \, \overrightarrow{V}_{\epsilon_1, \epsilon_2}(\xi)\Big\vert^2 d\xi}^{\frac{1}2}\\
	&= 2 \; \Big\Vert \!\! \sum_{\epsilon_1, \epsilon_2 \in \set{-1, 1}} C_{\epsilon_1, \epsilon_2}(\xi) \, \overrightarrow{V}_{\epsilon_1, \epsilon_2}(\xi)\; \Big\Vert_{L^2_\xi}.
\end{align*}
Recall that,  
$$\widehat{U_0}(\xi) = \sum_{\epsilon_1, \epsilon_2 \in \set{-1, 1}} C_{\epsilon_1, \epsilon_2}(\xi) \, \overrightarrow{V}_{\epsilon_1, \epsilon_2}(\xi).$$ Then, using Plancherel theorem, we finally obtain
\begin{equation*}
	\norm{\bU}_{L^p\pare{\R_+; L^q\pare{\R^3}}} \leqslant C R^{\frac{3}{2} - \frac{3}{q}} r^{-\frac{4}{p}} \eps^{\frac{1}{p}} \norm{U_0}_{L^2\pare{\RR^3}}.
\end{equation*}
Theorem \ref{th:Strichartz} is then proved. \hfill $\square$

\begin{rem}
	In what follows, we will only use a particular case of \eqref{eq:Strichartz} with $(p,q) = (4,+\infty)$
	\begin{equation}
		\label{eq:Strichartz4}
		\norm{\bU}_{L^4\pare{\R_+; L^\infty \pare{\R^3}}} \leqslant C R^{\frac{3}{2}} r^{-1} \varepsilon^{\frac{1}{4}} \norm{U_0}_{\2}.
	\end{equation}
\end{rem}

\section{The nonlinear part}\label{section:HFsystem}

In this section, we decompose the local solution $U^\varepsilon= \ub$ of \eqref{weakly_compressible_E} into two parts $$U^\varepsilon=\bU + \tU,$$ where $\bU$ is the global solution of \eqref{eq:free_wave} and $\tU$ solves (locally) the system
\begin{equation}\label{hi-freq_system}
	\left\lbrace
	\begin{aligned}
		&\partial_t \tU +\frac{1}{\varepsilon}\B \tU = \begin{pmatrix}
			-u^\eps \cdot \nabla u^\eps - \bgam b^\eps \nabla b^\eps\\
			- u^\eps \nabla b^\eps- \bgam b^\eps \dive u^\varepsilon
		\end{pmatrix}\\
		&\left. \tU \right|_{t=0}= \widetilde{U}_0^\eps = \left( 1-\Prr \right) U_0^\eps.
	\end{aligned}
	\right.
\end{equation}
As proven in Section \ref{se:Strichartz}, the linear system \eqref{eq:free_wave} is globally well-posed in $L^\infty(\R_+, \Hs^4)$, for any $s \geqslant 0$, and its solution goes to zero as $\eps \to 0$ in some $ L^p \left( \R_+; L^q \right) $-norm. On the contrary, the system \eqref{hi-freq_system} is a 3D nonlinear hyperbolic system, for which we can only expect to prove a long time existence of the solutions in $\Hs$, $s > 1 + \frac{3}2$, in the sense that, there exist $T^\eps \to +\infty$ as $\eps \to 0$, such that, $\tU \in L^{\infty}\pare{[0,T^\eps],\Hs^4}$. The main goal of this section is to prove the following theorem, which gives a lower bound estimate for the lifespan of $\tU$.

\begin{theorem}
	\label{eq:hi_freq}
	Let $s> \frac{5}{2}$, $s_0>0$, $1 < p < 2$. For any $0 < \eps < 1$ and for any $U_0^\eps \in Y_{s,s_0,p}$, where $Y_{s,s_0,p}$ is defined in \eqref{eq:DefY} and \eqref{eq:normY}, there exist $T_\eps^\star > 0$ and a unique solution $\tU$ to the system \eqref{hi-freq_system}, with initial data $\widetilde{U}_0^\eps$, satisfying $\tU \in C\pare{[0,T^\star_\varepsilon], \Hs^4}$. Moreover, there exist constants $\;\overline{C} > 0$, $\alpha > 0$ and $\eps^\star > 0$ such that, for any $\eps \in ]0,\eps^\star]$,
	$$
	T^\star_\varepsilon \geqslant \frac{\overline{C}}{\mathcal{C}(U_0^\eps)\, \eps^{\alpha}},
	$$
	where $\mathcal{C}(U_0^\eps)$ is defined in \eqref{eq:C0}. In addition, we can choose $\beta > 0$ such that the asymptotic behavior of $\tU$ when $\eps \to 0$ is determined as follows
	$$
	\left\| \tU \right\|_{L^\infty\left( \left[0,T^\star_\varepsilon\right],H^s \right)}=\mathcal{O}\left( \varepsilon^{\beta s} \right).
	$$
\end{theorem}

\bigskip

The proof of Theorem \ref{eq:hi_freq} will be divided into two parts

\begin{itemize}
	\item[\textbf{Part 1}] In the first part, using an iterative scheme, we prove that, for $\varepsilon_0$ small enough (which will be precised later), and for any $\varepsilon \in ]0,\varepsilon_0[$, there exists a unique strong solution $\tU$ of \eqref{hi-freq_system} in $L^\infty \left( \left[0,T\right],\Hs \right)$, with the lifespan $T > 0$ independent of $\varepsilon$. 

	\bigskip

	\item[\textbf{Part 2}] In the second part, we prove more refined estimates which allow to obtain the lower bound for the maximal lifespan $ T^\star_\varepsilon$ given in Theorem \ref{eq:hi_freq}.
\end{itemize}

\subsection{Local-in-time existence result for the nonlinear part.}

\label{sec:local_existence}

Throughout this part, we will always fix constants $\beta, \delta > 0$ and the Rossby number $\eps > 0$. We also set the radii of the frequency cut-off to be 
\begin{equation}
	\label{eq:radii} R = \eps^{-\beta}, \quad r = R^{-\delta} = \eps^{\beta\delta}.
\end{equation}
Our choice of these parameters will be explained and precised during the proof, at the place where we need to adjust their values.

Our goal is to prove the existence of a unique, local strong solution of the system \eqref{hi-freq_system}. To simplify the notations and the calculations, we rewrite \eqref{hi-freq_system} as follows 
\begin{equation}
	\label{system_local_existence}
	\left\lbrace
	\begin{aligned}
	& \partial_t \tU + \frac{1}{\varepsilon}\B \tU = - u^\varepsilon \cdot \nabla  U^\varepsilon -\bgam \left( 
	\begin{array}{c}
		b^\varepsilon \nabla b^\varepsilon\\
		b^\varepsilon \dive u^\varepsilon
	\end{array}	
	 \right)\\
	& \left. \tU \right|_{t=0}= \left( 1-\Prr \right)U_0
	\end{aligned}
\right.
\end{equation}
where we set $U^\varepsilon= \bU + \tU$ and where $ \bU $ is the solution of \eqref{eq:free_wave} . Our approach can be resumed in the following steps
\begin{enumerate}
	\item We introduce a sequence of linear systems,  indexed by $n\in\NN$, starting from \eqref{system_local_existence} and by induction with respect to $n$, we construct a solution of the $n$-th system defined in $\widetilde{L}^\infty \pare{\left[0, T_n\right], H^\sigma\pare{\RR^3}^4}$, for some given $\sigma \in ]s,s+s_0[$ and for some $T_n > 0$. We emphasize that $\eps$ will always be fixed during this stage.

	\bigskip

	\item We prove that we can choose $\eps_0 > 0$ small enough such that, for any $\eps \in ]0,\eps_0[$, the sequence of solutions of the previously introduced linear systems are uniformly bounded in $\widetilde{L}^\infty \left( \left[0, T_\eps \right]; H^\sigma(\R^3)^4 \right)$, for some $T_\eps > 0$ independent of $n$.

	\bigskip

	\item We prove that the sequence of solutions is a Cauchy sequence in $\widetilde{L}^\infty \pare{[0, T_\eps]; H^s(\R^3)^4}$. 

	\bigskip

	\item We check that the limit $\widetilde{U}^\eps$ satisfies \eqref{hi-freq_system}.
\end{enumerate}
The main technique result needed for our approach consists in the control of the bilinear terms. This control is given in the following lemma, which will be proven in the appendix.
\begin{lemma}
	\label{le:bi_bootstrap}
	The following estimates hold
	\begin{multline}
		\label{eq:bi_bootstrap}
		\int_0^t \abs{\psca{\tq \pare{u^\varepsilon(\tau) \cdot \nabla U^\varepsilon(\tau)} \;\big\vert\; \tq \tU(\tau)}_{L^2}} \, d\tau\\ 
		\leqslant C \;\mathcal{C}(U_0^\eps) b_q 2^{-2qs} R^{\frac{5+2\delta}2} t^{\frac{3}4} \eps^{\frac{1}4}  + C \;\mathcal{C}(U_0^\eps) b_q 2^{-2qs}\pare{R^{\frac{5+2\delta}2} t^{\frac{3}4} \eps^{\frac{1}4} + t \norm{\tU}_{\widetilde{L}^\infty\pare{[0,t], H^s}}} \norm{\tU}_{\widetilde{L}^\infty\pare{[0,t], H^s}}^2,
	\end{multline}
	\begin{multline}
		\label{eq:bi_bootstrap2}
		\int_0^t \abs{\psca{\tq \pare{b^\varepsilon (\tau) \nabla b^\varepsilon (\tau) } \;\big\vert\; \tq \widetilde{u}^\varepsilon (\tau)}_{L^2} + \psca{\tq \left( b^\varepsilon(\tau) \dive u^\varepsilon(\tau) \right) \;\big\vert\; \tq \widetilde{b}^\varepsilon (\tau) } } \, d\tau\\ 
		\leqslant C \;\mathcal{C}(U_0^\eps) b_q 2^{-2qs} R^{\frac{5+2\delta}2} t^{\frac{3}4} \eps^{\frac{1}4}  + C \;\mathcal{C}(U_0^\eps) b_q 2^{-2qs}\pare{R^{\frac{5+2\delta}2} t^{\frac{3}4} \eps^{\frac{1}4} + t \norm{\tU}_{\widetilde{L}^\infty\pare{[0,t], H^s}}} \norm{\tU}_{\widetilde{L}^\infty\pare{[0,t], H^s}}^2,
	\end{multline}
	where $b_q$ is a summable sequence such that $\sum_q b_q = 1$.
\end{lemma}

\bigskip

\noindent \textit{\textbf{Step 1.}} 
We fix $\sigma \in ]s,s+s_0[$, say $\sigma = (s + \frac{s_0}2)$. For any $n \in \NN^*$, we define the operator
\begin{equation*}
	J_n f = \mathcal{F}^{-1} \left( 1_{\left\{ \abs{\xi} \leqslant n\right\} \cap \left\{ \abs{\xi_h} \geqslant 1/n \right\} \cap \left\{ \abs{\xi_3} \geqslant 1/n\right\} } \widehat{f} \right)
\end{equation*}
which is continuous from $\2$ to $\2$. Setting $\widetilde{U}^0=0$, by induction, we define the following family of linear systems, related to \eqref{system_local_existence}
\begin{equation}
	\label{eq:seqlinsys_n} \tag{\ref{system_local_existence}${}_n$}
	\left\lbrace
	\begin{aligned}
	& \partial_t \widetilde{U}^{n+1}+\frac{1}{\varepsilon}\B \widetilde{U}^{n+1} + J_{n+1}\left(  \mathcal{A} \left( \bU + \widetilde{U}^n, D \right)  J_{n+1} \widetilde{U}^{n+1} \right) = - J_{n+1}\left(  \mathcal{A} \left( \bU + \widetilde{U}^n, D \right)  J_{n+1} \bU \right)\\
	&{\widetilde{ U}^{n+1}}_{|_{t=0}}= \widetilde{U}^{n+1}_0= \mathcal{F}^{-1} \left( 1_{B\left( 0,n+1 \right)} \mathcal{F}\left( 1-\mathcal{P}_{r,R} \right){U}_0 \right)
\end{aligned}
\right.
\end{equation}
$ \mathcal{A} $ is defined in \eqref{eq:defA}. We remark that since $\widetilde{U}^0\equiv 0$, $\widetilde{U}^1$ is solution to the following linear system
\begin{equation}
	\tag{\ref{system_local_existence}${}_0$}
	\left\lbrace
	\begin{aligned}
	& \partial_t \widetilde{U}^1+\frac{1}{\varepsilon} \B \widetilde{U}^1 + J_1\left(\mathcal{A} \left( \bU, D \right)  J_1 \widetilde{U}^1 \right) = - J_1\left(\mathcal{A} \left( \bU, D \right)  J_1 \bU \right)\\
	&{\widetilde{ U}^1}_{|_{t=0}}= \widetilde{U}^1_0= \mathcal{F}^{-1} \left( 1_{B\left( 0,1 \right)} \mathcal{F}\left( 1-\mathcal{P}_{r,R} \right){U}_0 \right)
\end{aligned}
\right.
\end{equation}
We have $\bU \in L^\infty\pare{\RR_+, H^\alpha(\RR^3)^4}$, for any $\alpha > 0$, because of its frequency localization property and Lemma \ref{le:smalll_data_Utilde} implies that $\widetilde{U}^1_0 \in H^\sigma(\RR^3)^4$. Then, we can easily construct the solution $\widetilde{U}^1$ of the linear system (\ref{system_local_existence}${}_0$) such that the support of $\mathcal{F}\pare{\widetilde{U}^1}$ is included in $B\left( 0,1 \right)$.

Now, for any $n \in \NN^*$, let $L^2_n\pare{\R^3}$ be the space 
\begin{equation*}
	L^2_n\pare{\R^3} = \set{f\in \2 \;\big\vert\; \text{Supp}\, \widehat{f} \subset    \left( \left\{ \left| \xi \right| \leqslant n\right\} \cap \left\{ \left| \xi_h \right| \geqslant 1/n\right\} \cap \left\{ \left| \xi_3 \right| \geqslant 1/n\right\} \right) }.
\end{equation*}
Let $\eta > 0$ be a fix positive constant and we suppose that, for any $0 \leqslant k\leqslant n-1$, we can construct a unique maximal solution $\widetilde{U}^{k+1}$ of (\ref{system_local_existence}${}_k$) in $$C^1\pare{[0,T_{k+1}], L^2_{k+1}\pare{\R^3}^4} \cap \widetilde{L}^\infty \pare{\left[0,T_{k+1}\right], H^\sigma\pare{\R^3}^4}$$ such that 
\begin{equation*}
	\left\| \widetilde{U}^{k+1} \right\|_{\widetilde{L}^\infty\left( \left[0,T_{k+1}\right]; H^\sigma \right)}\leqslant \eta.
\end{equation*}
Thanks to the embedding $H^\sigma(\R^3) \hra L^\infty(\R^3)$, we have $\bU + \widetilde{U}^n \in L^\infty\pare{[0,T_n],\Linfty^4}$, which implies that 
$$J_{n+1}\left(\mathcal{A} \left( \bU + \widetilde{U}^n\left( t \right), D \right) J_{n+1} \varphi \right)\in \2^4, \quad \forall\, \varphi \in \2^4$$ 
Indeed, since the Fourier transform of $J_{n+1} \varphi$ is supported in $\mathcal{B}_3(0,n+1)$, using Bernstein lemma \ref{lemma:Bernstein}, we have
\begin{align*}
	\norm{J_{n+1}\pare{\mathcal{A} \left( \bU + \widetilde{U}^n(t), D \right) J_{n+1} \varphi}}_{L^2} &\leqslant \norm{\bU + \widetilde{U}^n(t)}_{L^\infty} \norm{\nabla J_{n+1}\varphi}_{L^2}\\ 
	&\leqslant \norm{\bU + \widetilde{U}^n(t)}_{L^\infty} (n+1) \norm{\varphi}_{L^2}.
\end{align*}

We can rewrite \eqref{eq:seqlinsys_n} as an ODE
$$
\partial_t \widetilde{U}^{n+1}= \mathcal{L}_{n+1} \widetilde{U}^{n+1} + \Phi,
$$
where $\Phi \in \2^4$ and the linear operator $\mathcal{L}_{n+1}$ maps continuously $\2^4$ to $\2^4$. The Cauchy-Lipschitz theorem ensure the existence of a unique maximal solution to the system \eqref{eq:seqlinsys_n} $$\widetilde{U}^{n+1} \in \mathcal{C}^1\left( \left[0,T_{n+1}\right]; \2^4 \right).$$ Moreover, since $ J_{n+1}^2=J_{n+1} $, applying $ J_{n+1} $ to \eqref{eq:seqlinsys_n}, we obtain, by uniqueness, that $$ J_{n+1} \widetilde{U}^{n+1} = \widetilde{U}^{n+1}.$$ Hence, $ \widetilde{U}^{n+1} $ belongs not only to $\2^4$ but to $L^2_{n+1}\pare{\R^3}^4$, which conclude the first step by induction.

\bigskip

\noindent \textit{\textbf{Step 2.}} We recall that throughout this paper, we use $C$ to denote a generic positive constant which can change from line to line. In this step, we want to prove that, for previously chosen $\eps > 0$ small enough, the sequence $\set{T_n}$ is bounded from below away from zero, which means that there exists $T_\eps > 0$ such that, for any $n \in \NN$,  
\begin{equation}
	\label{eq:boundtildeUn} \left\| \widetilde{U}^{n} \right\|_{\widetilde{L}^\infty\left( \left[0,T_\eps\right]; H^\sigma \right)} \leqslant \eta.
\end{equation}

We will prove \eqref{eq:boundtildeUn} by induction. For $n = 0$, we have nothing to do. So we suppose that, for fix $T_\eps > 0$ which will be precised later, \eqref{eq:boundtildeUn} is true for any $0 \leqslant k \leqslant n$. Now, we want to estimate $\widetilde{U}^{n+1}$ in $\widetilde{L}^\infty\pare{[0,T_\eps], H^\sigma\pare{\RR^3}}$-norm. Applying $\tq$ to \eqref{eq:seqlinsys_n}, taking the $L^2$-scalar product of the obtained equation with $\tq \widetilde{U}^{n+1}$ and then integrating with respect to the time variable on $[0,t]$, we get
\begin{align}
	\label{eq:tildeUn01} \norm{\tq \widetilde{U}^{n+1} (t)}_{L^2}^2 \leqslant \norm{\tq \widetilde{U}^{n+1}_0}_{L^2}^2 &+ 2 \int_0^t \abs{\psca{J_{n+1} \tq \pare{u^n \cdot \nabla U^{n+1}}} \;\big\vert\; \tq \widetilde{U}^{n+1}} (\tau) \, d\tau\\
	&+2 \int_0^t \Big|\psca{J_{n+1}\tq \pare{b^n \nabla b^{n+1}} \;\big\vert\; \tq \widetilde{u}^{n+1}}_{L^2} + \notag\\
	& \qquad \qquad \qquad + \psca{J_{n+1}\tq \left( b^n \dive u^{n+1} \right) \;\big\vert\; \tq \widetilde{b}^{n+1}} \Big| (\tau) \, d\tau \notag
\end{align}
Using the same method as in the proof of Lemma \ref{le:bi_bootstrap}, we decompose the bilinear term on the right hand side of \eqref{eq:tildeUn01} into the following sums
\begin{equation*}
	\int_0^t \abs{\psca{J_{n+1} \tq \pare{u^n \cdot \nabla U^{n+1}}} \;\big\vert\; \tq \widetilde{U}^{n+1}} (\tau) \, d\tau \leqslant B^{n+1}_1 + B^{n+1}_2 + B^{n+1}_3 + B^{n+1}_4,
\end{equation*}
and
\begin{multline*}
	\int_0^t \left|\psca{J_{n+1}\tq \pare{b^n \nabla b^{n+1}} \;\big\vert\; \tq \widetilde{u}^{n+1}}_{L^2} + \psca{J_{n+1}\tq \left( b^n \dive u^{n+1} \right) \;\big\vert\; \tq \widetilde{b}^{n+1}} \right| (\tau) \, d\tau\\
	\leqslant C^{n+1}_1 + C^{n+1}_2 + C^{n+1}_3 + C^{n+1}_4.
\end{multline*}
where, applying Lemmas \ref{le:biproduct}, \ref{le:biprod_cut2} and \ref{le:biprod_cut1} in the appendix, we have
\begin{align}
	\label{eq:tildeBn01} B^{n+1}_1 &= \int_0^t \abs{\psca{\tq \pare{\overline{u}^\eps \cdot \nabla \bU}} \;\big\vert\; \tq \widetilde{U}^{n+1}} (\tau) \, d\tau\\
	&\leqslant C \norm{U_0^\eps}_{H^\sigma} R^{\frac{5+2\delta}2} t^{\frac{3}4} \eps^{\frac{1}4} b_q 2^{-2q\sigma} \norm{\bU}_{\widetilde{L}^\infty\pare{[0,t], H^\sigma}} \norm{\widetilde{U}^{n+1}}_{\widetilde{L}^\infty\pare{[0,t], H^\sigma}} \notag\\
	&\leqslant C \;\mathcal{C}(U_0^\eps) \, t^{\frac{3}4} \eps^{\frac{1}4-\frac{\beta(5+2\delta)}2} b_q 2^{-2q\sigma} \pare{1 + \norm{\widetilde{U}^{n+1}}_{\widetilde{L}^\infty\pare{[0,t], H^\sigma}}^2}, \notag
\end{align}
\begin{align}
	\label{eq:tildeBn02} B^{n+1}_2 &= \int_0^t \abs{\psca{\tq \pare{\widetilde{u}^n \cdot \nabla \bU}} \;\big\vert\; \tq \widetilde{U}^{n+1}} (\tau) \, d\tau\\
	&\leqslant C \;\mathcal{C}(U_0^\eps) \, R^{\frac{5+2\delta}2} t^{\frac{3}4} \eps^{\frac{1}4} b_q 2^{-2q\sigma} \norm{\widetilde{U}^n}_{\widetilde{L}^\infty\pare{[0,t], H^\sigma}} \norm{\widetilde{U}^{n+1}}_{\widetilde{L}^\infty\pare{[0,t], H^\sigma}} \notag\\
	&\leqslant C \;\mathcal{C}(U_0^\eps) \, \eta \, t^{\frac{3}4} \eps^{\frac{1}4-\frac{\beta(5+2\delta)}2} b_q 2^{-2q\sigma} \pare{1 + \norm{\widetilde{U}^{n+1}}_{\widetilde{L}^\infty\pare{[0,t], H^\sigma}}^2}, \notag
\end{align}
\begin{align}
	\label{eq:tildeBn03} B^{n+1}_3 &= \int_0^t \abs{\psca{\tq \pare{\overline{u}^\eps \cdot \nabla \widetilde{U}^{n+1}}} \;\big\vert\; \tq \widetilde{U}^{n+1}} (\tau) \, d\tau\\
	&\leqslant C \;\mathcal{C}(U_0^\eps) \, R^{\frac{5+2\delta}2} t^{\frac{3}4} \eps^{\frac{1}4} b_q 2^{-2q\sigma} \norm{\widetilde{U}^{n+1}}_{\widetilde{L}^\infty\pare{[0,t], H^\sigma}}^2 \qquad \qquad \qquad \quad \notag\\
	&\leqslant C \;\mathcal{C}(U_0^\eps) \, t^{\frac{3}4} \eps^{\frac{1}4-\frac{\beta(5+2\delta)}2} b_q 2^{-2q\sigma} \norm{\widetilde{U}^{n+1}}_{\widetilde{L}^\infty\pare{[0,t], H^\sigma}}^2, \notag
\end{align}
\begin{align}
	\label{eq:tildeBn04} B^{n+1}_4 &= \int_0^t \abs{\psca{\tq \pare{\widetilde{u}^n \cdot \nabla \widetilde{U}^{n+1}}} \;\big\vert\; \tq \widetilde{U}^{n+1}} (\tau) \, d\tau \qquad \\
	&\leqslant C b_q 2^{-2q\sigma} \norm{\widetilde{U}^n}_{\widetilde{L}^\infty([0,t],H^\sigma)} \norm{\widetilde{U}^{n+1}}_{\widetilde{L}^2([0,t],H^\sigma)}^2 \notag  \qquad \qquad \qquad \quad \\
	&\leqslant C \eta \, t \, b_q 2^{-2q\sigma} \norm{\widetilde{U}^{n+1}}_{\widetilde{L}^\infty([0,t],H^\sigma)}^2, \notag
\end{align}
and
\begin{align}
	\label{eq:tildeCn01} C^{n+1}_1 &= \int_0^t \abs{\psca{\tq \pare{\overline{b}^n \nabla \overline{b}^{n+1}} \big\vert \tq \widetilde{u}^{n+1}}} (\tau) \, d\tau + \int_0^t \abs{\psca{\tq \pare{\overline{b}^n \dive \overline{u}^{n+1}} \big\vert \tq \widetilde{b}^{n+1}}} (\tau) \, d\tau\\
	&\leqslant C \norm{U_0^\eps}_{H^\sigma} R^{\frac{5+2\delta}2} t^{\frac{3}4} \eps^{\frac{1}4} b_q 2^{-2q\sigma} \norm{\bU}_{\widetilde{L}^\infty\pare{[0,t], H^\sigma}} \norm{\widetilde{U}^{n+1}}_{\widetilde{L}^\infty\pare{[0,t], H^\sigma}} \notag\\
	&\leqslant C \;\mathcal{C}(U_0^\eps) \, t^{\frac{3}4} \eps^{\frac{1}4-\frac{\beta(5+2\delta)}2} b_q 2^{-2q\sigma} \pare{1 + \norm{\widetilde{U}^{n+1}}_{\widetilde{L}^\infty\pare{[0,t], H^\sigma}}^2}, \notag
\end{align}
\begin{align}
	\label{eq:tildeCn02} C^{n+1}_2 &= \int_0^t \abs{\psca{\tq \pare{\widetilde{b}^n \nabla \overline{b}^{n+1}} \big\vert \tq \widetilde{u}^{n+1}}} (\tau) \, d\tau + \int_0^t \abs{\psca{\tq \pare{\widetilde{b}^n \dive \overline{u}^{n+1}} \big\vert \tq \widetilde{b}^{n+1}}} (\tau) \, d\tau\\
	&\leqslant C \;\mathcal{C}(U_0^\eps) \, R^{\frac{5+2\delta}2} t^{\frac{3}4} \eps^{\frac{1}4} b_q 2^{-2q\sigma} \norm{\widetilde{U}^n}_{\widetilde{L}^\infty\pare{[0,t], H^\sigma}} \norm{\widetilde{U}^{n+1}}_{\widetilde{L}^\infty\pare{[0,t], H^\sigma}} \notag\\
	&\leqslant C \;\mathcal{C}(U_0^\eps) \, \eta \, t^{\frac{3}4} \eps^{\frac{1}4-\frac{\beta(5+2\delta)}2} b_q 2^{-2q\sigma} \pare{1 + \norm{\widetilde{U}^{n+1}}_{\widetilde{L}^\infty\pare{[0,t], H^\sigma}}^2}, \notag
\end{align}
\begin{align}
	\label{eq:tildeCn03} C^{n+1}_3 &= \int_0^t \abs{\psca{\tq \pare{\overline{b}^n \nabla \widetilde{b}^{n+1}} \big\vert \tq \widetilde{u}^{n+1}}} (\tau) \, d\tau + \int_0^t \abs{\psca{\tq \pare{\overline{b}^n \dive \widetilde{u}^{n+1}} \big\vert \tq \widetilde{b}^{n+1}}} (\tau) \, d\tau\\
	&\leqslant C \;\mathcal{C}(U_0^\eps) \, R^{\frac{5+2\delta}2} t^{\frac{3}4} \eps^{\frac{1}4} b_q 2^{-2q\sigma} \norm{\widetilde{U}^{n+1}}_{\widetilde{L}^\infty\pare{[0,t], H^\sigma}}^2 \qquad \qquad \qquad \quad \notag\\
	&\leqslant C \;\mathcal{C}(U_0^\eps) \, t^{\frac{3}4} \eps^{\frac{1}4-\frac{\beta(5+2\delta)}2} b_q 2^{-2q\sigma} \norm{\widetilde{U}^{n+1}}_{\widetilde{L}^\infty\pare{[0,t], H^\sigma}}^2, \notag
\end{align}
\begin{align}
	\label{eq:tildeCn04} C^{n+1}_4 &= \int_0^t \left|\psca{\tq \pare{\widetilde{b}^n \nabla \widetilde{b}^{n+1}} \;\big\vert\; \tq \widetilde{u}^{n+1}} + \psca{\tq \pare{\widetilde{b}^n \dive \widetilde{u}^{n+1}} \;\big\vert\; \tq \widetilde{b}^{n+1}} \right| (\tau) \, d\tau \\
	&\leqslant C b_q 2^{-2q\sigma} \norm{\widetilde{U}^n}_{\widetilde{L}^\infty([0,t],H^\sigma)} \norm{\widetilde{U}^{n+1}}_{\widetilde{L}^2([0,t],H^\sigma)}^2 \notag\\
	&\leqslant C \eta \, t \, b_q 2^{-2q\sigma} \norm{\widetilde{U}^{n+1}}_{\widetilde{L}^\infty([0,t],H^\sigma)}^2, \notag
\end{align}
where $\mathcal{C}(U_0^\eps)$ is defined in \eqref{eq:C0}.

\bigskip

Inserting Estimates \eqref{eq:tildeBn01} to \eqref{eq:tildeCn04} into \eqref{eq:tildeUn01}, multiplying the obtained inquality by $2^{2q\sigma}$, then summing with respect to $q \geqslant -1$ and applying Lemma \ref{le:smalll_data_Utilde} lead to
\begin{multline*}
	\norm{\widetilde{U}^{n+1}}_{\widetilde{L}^\infty([0,t],H^\sigma)}^2 \leqslant C \;\mathcal{C}(U_0) \, \eps^{2\beta s_0} + C \eta \, t \norm{\widetilde{U}^{n+1}}_{\widetilde{L}^\infty([0,t],H^\sigma)}^2 \\ + C \;\mathcal{C}(U_0^\eps) \, t^{\frac{3}4} \eps^{\frac{1}4-\frac{\beta(5+2\delta)}2} \pare{1 + \norm{\widetilde{U}^{n+1}}_{\widetilde{L}^\infty\pare{[0,t], H^\sigma}}^2}.\qquad 
\end{multline*}
We recall that $\sigma \in ]s,s+s_0[$ and $\eta > 0$ are fixed positive constants and $\delta > 0$ is given in Lemma \ref{le:smalll_data_Utilde}. We choose $\eps_0 > 0$, $T_\eps > 0$ and $\beta > 0$ such that
\begin{equation}
	\label{eq:Conditions01}
	\left\{ 
	\begin{aligned}
		&\beta(5+2\delta) < \frac{1}2\\
		&C \eta T_\eps +  C \;\mathcal{C}(U_0^\eps) \, T_\eps^{\frac{3}4} \eps_0^{\frac{1}4-\frac{\beta(5+2\delta)}2} < \frac{1}2\\
		&C \;\mathcal{C}(U_0^\eps) \, \eps_0^{2\beta s_0} + C \;\mathcal{C}(U_0^\eps) \, T_\eps^{\frac{3}4} \eps_0^{\frac{1}4-\frac{\beta(5+2\delta)}2} \leqslant \frac{\eta^2}2.
	\end{aligned}
	\right.
\end{equation}
Then, for any $\eps \in [0,\eps_0]$, we deduce that,
\begin{equation*}
	\norm{\widetilde{U}^{n+1}}_{\widetilde{L}^\infty\pare{[0,T_\eps], H^\sigma}} \leqslant \eta,
\end{equation*}
and Step 2 is concluded.

\begin{rem}
	In fact, the time of existence $T_\eps = T > 0$ depends only on $\eps_0 > 0$ and thus is independent of $\eps$, for $\eps \in ]0,\eps_0[$.
\end{rem}

\bigskip

\noindent \textit{\textbf{Step 3.}} At first, we will prove that $\set{\widetilde{U}^{n}}_n$ is a Cauchy sequence in the space $L^\infty\left( \left[0,T\right]; \2^4 \right)$. We define the auxiliary sequence $\set{\widetilde{V}^n}_n $ by
$$ \widetilde{V}^{n+1}= \widetilde{U}^{n+1}-\widetilde{U}^n, \quad \forall\; n\in\NN.$$ For any $n \in \NN$, $\widetilde{V}^{n+1}$ is solution of the system
\begin{equation}
	\label{eq:deltautilde}
	\left\lbrace
	\begin{aligned}
		& \partial_t\widetilde{V}^{n+1} + \frac{1}{\varepsilon}\B \widetilde{V}^{n+1} 
		+\mathcal{A} \left( \widetilde{U}^n, D \right) \widetilde{V}^{n+1} 
		 + \mathcal{A} \left( \widetilde{V}^n, D \right) \widetilde{U}^n 
		 + \mathcal{A} \left( \widetilde{V}^n, D \right) \bU
		 + \mathcal{A} \left( \bU, D \right) \widetilde{V}^{n+1}
		  = 0\\
		& \left. \widetilde{V}^{n+1} \right|_{t=0}=  \widetilde{U}^{n+1}_0 -  \widetilde{U}^{n}_0.
	\end{aligned}
	\right.
\end{equation}
We will need the following estimates, the proof of which is simple and direct.
\begin{lemma}
	\label{bilinear_delta} The following estimates hold
	\begin{align*}
		\abs{\psca{\mathcal{A} \left( \widetilde{U}^n, D \right) \widetilde{V}^{n+1} \;\Big\vert\; \widetilde{V}^{n+1}}_{L^2}} &\leqslant C \left\| \nabla \widetilde{U}^n \right\|_{L^\infty}\left\| \widetilde{V}^{n+1} \right\|^2_{L^2}\\
		\abs{\psca{\mathcal{A} \left( \widetilde{V}^n, D \right) \widetilde{U}^n  \;\Big\vert\; \widetilde{V}^{n+1}}_{L^2}} &\leqslant C \left\| \nabla \widetilde{U}^n \right\|_{L^\infty}\left\| \widetilde{V}^{n} \right\|_{L^2}\left\| \widetilde{V}^{n+1} \right\|_{L^2}\\
		\abs{\psca{\mathcal{A} \left( \widetilde{V}^n, D \right) \bU \;\Big\vert\; \widetilde{V}^{n+1}}_{L^2}} &\leqslant C \left\| \nabla \bU \right\|_{L^\infty}\left\| \widetilde{V}^{n} \right\|_{L^2}\left\| \widetilde{V}^{n+1} \right\|_{L^2}\\
		\abs{\psca{\mathcal{A} \left( \bU, D \right) \widetilde{V}^{n+1} \;\Big\vert\; \widetilde{V}^{n+1}}_{L^2}} &\leqslant C \left\| \nabla \bU \right\|_{L^\infty}\left\| \widetilde{V}^{n+1} \right\|^2_{L^2}.
	\end{align*}
\end{lemma}

Taking the $L^2$ scalar product of the first equation of \eqref{eq:deltautilde} with $\norm{\widetilde{V}^{n+1}}_{L^2}^{-1} \widetilde{V}^{n+1}$ and using Bernstein Lemma \ref{lemma:Bernstein}, Lemma \ref{bilinear_delta} and the Sobolev inclusion $H^\sigma\pare{\RR^3} \hookrightarrow W^{1,\infty}\pare{\RR^3}$, we obtain
\begin{equation}
	\label{eq:tildeVn1}
	\frac{d}{dt} \norm{\widetilde{V}^{n+1}}_{L^2} \leqslant \pare{\norm{\widetilde{U}^n}_{H^\sigma} + R \norm{\bU}_{L^\infty}} \pare{\norm{\widetilde{V}^{n+1}}_{L^2} + \norm{\widetilde{V}^n}_{L^2}}.
\end{equation}
We recall that, for any $n \in \NN$, 
\begin{equation*}
	\left\| \widetilde{U}^n \right\|_{\widetilde{L}^\infty \left( \left[0,T \right]; H^\sigma \right) }\leqslant \eta
\end{equation*}
and that H\"older inequality and Estimate \eqref{eq:Strichartz4} give, for any $0 < t \leqslant T$,
\begin{equation*}
	\int_0^t \left\| \bU \right\|_{L^\infty} ds \leqslant \pare{\int_0^t ds}^{\frac{3}4} \; \norm{\bU}_{L^4\pare{[0,t],L^\infty\pare{\RR^3}}} \leqslant C \;\mathcal{C}(U_0^\eps) T^{\frac{3}4} R^{\frac{3+2\delta}2} \eps^{\frac{1}4},
\end{equation*}
where $\mathcal{C}(U_0^\eps)$ is defined in \eqref{eq:C0}. Using Bernstein Lemma \ref{lemma:Bernstein}, we also have
\begin{equation*}
	\norm{\widetilde{V}^{n+1}(0)}_{L^2}^2 = \norm{\widetilde{U}^{n+1}_0 -  \widetilde{U}^{n}_0}_{L^2}^2 \leqslant \int_{\abs{\xi} > n} \pare{1+\abs{\xi}^2}^{-s} \pare{1+\abs{\xi}^2}^s \abs{\widehat{U}_0}^2 d\xi \leqslant C \;\mathcal{C}(U_0^\eps) n^{-2s}.
\end{equation*}
Integrating \eqref{eq:tildeVn1} with respect to the time variable and taking into account all the above inequalities and remarking that we already chose $R = \eps^{-\beta}$, $\beta > 0$ (see formula \eqref{eq:radii}), we obtain
\begin{equation*}
	v_{n+1} \leq C \;\mathcal{C}(U_0^\eps) n^{-s} + \pare{\eta T + C\;\mathcal{C}(U_0^\eps) T^{\frac{3}4} \eps^{\frac{1}4 - \frac{\beta(5+2\delta)}2}} \pare{v_{n+1} + v_n},
\end{equation*}
where for any $n\in\NN$, we set
\begin{equation*}
	v_n = \norm{\widetilde{V}^n}_{L^\infty\pare{[0,T],L^2}}.
\end{equation*}
If we choose the parameters such that
\begin{equation}
	\label{eq:Conditions02}
	\left\{ 
	\begin{aligned}
		&\beta(5+2\delta) < \frac{1}2\\
		&\eta T + C\;\mathcal{C}(U_0^\eps) \, T^{\frac{3}4} \eps_0^{\frac{1}4 - \frac{\beta(5+2\delta)}2} < \frac{1}3,
	\end{aligned}
	\right.
\end{equation}
then, for any $\eps \in ]0,\eps_0[$, we have
\begin{equation*}
	v_{n+1} \leq C \;\mathcal{C}(U_0^\eps) \, n^{-s} + \frac{1}2 v_n.
\end{equation*}
Since $s > \frac{5}2$, the series $\sum_{n\in \NN} n^{-s}$ is convergent, which implies that the sequence $\set{v^n}_n$ is summable, which in turn implies that $\set{\widetilde{U}^n}_n$ is a Cauchy sequence in $L^\infty\pare{[0,T], \2^4}$. Since $\set{\widetilde{U}^n}_n$ is a bounded sequence in $\widetilde{L}^\infty\pare{[0,T], H^\sigma\pare{\RR^3}^4}$, for some $\sigma \in ]s,s+s_0[$, by interpolation, we deduce that $\set{\widetilde{U}^n}_n$ is a Cauchy sequence in $\widetilde{L}^\infty\pare{[0,T], H^s\pare{\RR^3}^4}$, and so, there exists $\tU$ in $\widetilde{L}^\infty\pare{[0,T], H^s\pare{\RR^3}^4}$ such that
\begin{equation*}
	\tU = \lim_{n\to +\infty} \widetilde{U}^n.
\end{equation*}

\begin{rem}
	Fixing $s > \frac{5}2$, $s_0 > 0$, $\sigma \in ]s,s+s_0[$, $\eta > 0$, $\delta > 0$, the conditions \eqref{eq:Conditions01} and \eqref{eq:Conditions02} can easily be satisfied by choosing $\beta > 0$, $T > 0$ and $\eps_0 > 0$ sufficiently small.
\end{rem}

\bigskip

\noindent \textit{\textbf{Step 4.}} It remains to verify if $\tU$ is a solution of \eqref{hi-freq_system}. In fact, we only have to check if we can pass to the limit in the bilinear term in the first equation of the system \eqref{system_local_existence}. Since $s > \frac{5}2$, classical product laws in Sobolev spaces yield
\begin{align*}
	&\left\| \mathcal{A}\left( \tU, D \right) \tU  - \mathcal{A}\left(\widetilde{U}^n, D\right)\widetilde{U}^{n+1}\right\|_{H^{s-1}}\\ 
	&\qquad \qquad \leqslant \left\|\mathcal{A}\left(  \widetilde{U}^n - \tU, D \right) \widetilde{U}^{n+1} \right\|_{H^{s-1}} + \left\| \mathcal{A}\left(\tU, D\right) \left(\widetilde{U}^{n+1} -  \tU \right) \right\|_{H^{s-1}}\\
	&\qquad \qquad \leqslant \norm{\widetilde{U}^n - \tU}_{H^{s-1}} \norm{\nabla \widetilde{U}^{n+1}}_{L^\infty} + \norm{\widetilde{U}^n - \tU}_{L^\infty} \norm{\widetilde{U}^{n+1}}_{H^s}\\ 
	&\qquad \qquad \quad + \norm{\tU}_{H^{s-1}} \norm{\nabla\pare{\widetilde{U}^{n+1} - \tU}}_{L^\infty} + \norm{\tU}_{L^\infty} \norm{\widetilde{U}^{n+1} - \tU}_{H^s}\\
	&\qquad \qquad \leqslant C\norm{\tU}_{H^s} \norm{\widetilde{U}^{n+1} - \tU}_{H^s} + C \norm{\widetilde{U}^{n+1}}_{H^s} \norm{\widetilde{U}^n - \tU}_{H^s}
\end{align*}
We recall that $\widetilde{U}^n$ and so $\tU$ are bounded in $L^\infty\pare{[0,T],\Hs^4}$ by $\eta > 0$. Besides, we also prove in Step 3 that $$\lim_{n\to +\infty} \norm{\widetilde{U}^n - \tU}_{L^\infty\pare{[0,T],H^s}} = 0.$$ Thus, we obtain 
$$\lim_{n\to +\infty} \left\| \mathcal{A}\left( \tU, D \right) \tU  - \mathcal{A}\left(\widetilde{U}^n, D\right)\widetilde{U}^{n+1}\right\|_{L^\infty\pare{[0,T],H^{s-1}}} = 0,$$ 
which allows to pass to the limit and conclude Step 4.

\bigskip

We remark that by construction, $\tU \in C\pare{[0,T], L^2(\RR^3)^4} \cap L^\infty\pare{[0,T],H^\sigma(\RR^3)^4}$, for some $\sigma > s$, which implies that $\tU \in C\pare{[0,T],\Hs^4}$. To finish this part, we study the uniqueness and the continuity with respect to the initial data of the previously contructed solution. More precisely, we prove the following lemma.

\begin{lemma}
Let $ U_0\in H^{s+s_0}, s>5/2, s_0>0 $. There exists a unique solution of the system \eqref{eq:WCEshortsym} in $ \widetilde{L }^{\infty} \left( \left[0, T\right]; \Hs^4 \right)$. Moreover, if $\Phi$ is the function which associates to $ U_0\in H^{s+s_0} $ the unique solution $ U $ of \eqref{eq:WCEshortsym}, then
$$
\Phi \in \mathcal{C}^{0,1} \left( \Hs^4; {L }^{\infty} \left( \left[0, T\right]; \Hs^4 \right) \right).
$$
\end{lemma}

\begin{proof}
Let us consider two initial data $ U_{i, 0}^\eps \in H^{s+s_0}, s>5/2, s_0>0, i=1,2 $. These data generate two solutions $U_i^\eps$, $i = 1,2$, to the system
\begin{equation*}
	\left\lbrace
	\begin{aligned}
		&\partial_t U_i^\eps -\frac{1}{\varepsilon}\mathcal{ B} U_i^\eps = - \mathcal{A} \left( U_i^\eps , D \right) U_i^\eps,\\ 
		& \left. U_i^\eps\right|_{t=0}= U_{i,0}^\eps.
	\end{aligned}
	\right.
\end{equation*}

We remark that $\delta U^\eps = U_1^\eps-U_2^\eps$ solves the system
\begin{equation}
	\label{eq:difference_system_uniqueness}
	\left\lbrace
	\begin{array}{l}
		\partial_t \delta U^\eps + \ds\frac{1}{\varepsilon} \B \delta U^\eps = -\mathcal{A}\left( \delta U^\eps , D \right) U_2^\eps - \mathcal{A}\left( U_1^\eps, D \right) \delta U^\eps\\
		\left. \delta U^\eps\right|_{t=0}= U_{1,0}^\eps-U_{2,0}^\eps.
	\end{array}
	\right.
\end{equation}
Taking $ \2 $ scalar product of the first equation of \eqref{eq:difference_system_uniqueness} with $\delta U^\eps$, and considering the following inequalities, 
\begin{align*}
	&\abs{ \psca{ \mathcal{A}\left( \delta U^\eps , D \right) U_2^\eps \;\Big\vert\; \delta U^\eps }_\2} \leqslant \left\| \nabla U_2^\eps \right\|_{\Linfty} \left\| \delta U^\eps \right\|_{L^2}^2\\
	&\abs{ \psca{ \mathcal{A}\left( U_1^\eps, D \right)\delta U^\eps \;\Big\vert\; \delta U^\eps }_\2} \leqslant \left\| \nabla U_1^\eps \right\|_{\Linfty} \left\| \delta U^\eps \right\|_{L^2}^2,
\end{align*}
we deduce via Gronwall inequality, and the embedding $ \Hs\hra \Woi $ that
\begin{align*}
	\left\| \delta U^\eps  \left( t \right)\right\|_\2^2 \leqslant \left\| \delta U^\eps(0) \right\|_\2^2 \
	e^{2 \int_0^t \left( \left\| U_1^\eps \left( \tau \right) \right\|_\Hs + \left\| U_2^\eps \left( \tau \right) \right\|_\Hs \right) d \tau}.
\end{align*}
From the construction of the solution, we have 
$$\left\| U_i^\eps \right\|_{\widetilde{L}^\infty \left( \left[0,T\right]; \Hs^4 \right) } \leqslant \left\| U_{i,0}^\eps \right\|_\Hs + \eta,$$ 
hence
\begin{align*}
	e^{2\int_0^t \left( \left\| U_1^\eps \left( \tau \right) \right\|_\Hs + \left\| U_2^\eps \left( \tau \right) \right\|_\Hs \right) d \tau} \leqslant e^{ \left( 4\eta + 2\left\| U_{1,0}^\eps \right\|_\Hs + 2\left\| U_{2,0}^\eps \right\|_\Hs \right)T},
\end{align*}
which implies the uniqueness and the continuity of the solution in the space $ L^\infty \left( [0,T];\2^4 \right) $. The uniqueness and the continuity in $ {L}^\infty \left( \left[0,T\right]; \Hs^4 \right) $ follow by interpolation.
\end{proof}

\subsection{Lifespan of the nonlinear part.}

In this part, we will provide a control of the maximal lifespan $T^\star_\eps$ of the solution previously constructed. Applying $\tq$ to \eqref{hi-freq_system}, taking the $L^2$-scalar product of the obtained equation with $\tq \tU$ and then integrating with respect to the time variable on $[0,t]$, we get
\begin{multline}
	\label{eq:tildeU01} \norm{\tq \tU (t)}_{L^2}^2 \leqslant \norm{\tq \widetilde{U}_0}_{L^2}^2 + 2 \int_0^t \abs{\psca{\tq \pare{u^\eps(\tau) \cdot \nabla U^\eps(\tau)} \;\big\vert\; \tq \tU(\tau)}} \, d\tau\\
	+ 2 \int_0^t \abs{\psca{\tq\left( \begin{array}{c} b^\varepsilon\nabla b^\varepsilon\\ b^\varepsilon \dive u^\varepsilon \end{array} \right)	\left( \tau \right) \;\big\vert\; \tq \tU (\tau)}_{L^2}} d\tau.
\end{multline}
We recall that $R = \eps^{-\beta}$, $\beta > 0$. Then, using Lemma \ref{le:smalll_data_Utilde}, we deduce the existence of $\eps_1 \in ]0,1[$ such that, for $0 < \eps \leqslant \eps_1$,
\begin{equation*}
	\norm{\widetilde{U}_0}_{H^s} \leqslant C\;\mathcal{C}(U_0^\eps) \, R^{-s_0} = C\;\mathcal{C}(U_0^\eps) \, \eps^{\beta s_0} \leqslant \eps^{\frac{\beta s_0}2},
\end{equation*}
where $\mathcal{C}(U_0^\eps)$ is defined in \eqref{eq:C0}. Let
\begin{equation}
	\label{eq:max_time} T_0 = \sup \set{T> 0 \;:\; \norm{\tU(t)}_{\widetilde{L}^\infty\pare{[0,t], H^s}} \leqslant 2 \eps^{\frac{\beta s_0}2}, \;\forall\, t \in [0,T]}.
\end{equation}
The continuity of $\tU$ with respect to the time variable implies that $T_0 > 0$. Multiplying \eqref{eq:tildeU01} by $2^{2qs}$ and summing with respect to $q \geqslant -1$, then using Lemma \ref{le:bi_bootstrap}, for any $t \in [0,T_0[$, we obtain
\begin{multline*}
	\norm{\tU}_{\widetilde{L}^\infty([0,t],H^s)}^2 \leqslant \norm{\widetilde{U}_0}_{H^s}^2 + C \;\mathcal{C}(U_0^\eps) \, t^{\frac{3}4} \eps^{\frac{1}4 - \frac{\beta(5+2\delta)}2}\\ + C \;\mathcal{C}(U_0^\eps) \, \pare{t^{\frac{3}4} \eps^{\frac{1}4 - \frac{\beta(5+2\delta)}2} + t \norm{\tU}_{\widetilde{L}^\infty\pare{[0,t], H^s}}} \norm{\tU}_{\widetilde{L}^\infty\pare{[0,t], H^s}}^2\notag
\end{multline*}
which implies that for any $0 < \eps \leqslant \eps_1$ and for any $t \in [0,T_0[$,
\begin{equation*}
	\norm{\tU}_{\widetilde{L}^\infty([0,t],H^s)}^2 \leqslant \eps^{\beta s_0} + C \;\mathcal{C}(U_0^\eps) \, t^{\frac{3}4} \eps^{\frac{1}4 - \frac{\beta(5+2\delta)}2} + C \;\mathcal{C}(U_0^\eps) \, \pare{t^{\frac{3}4} \eps^{\frac{1}4 - \frac{\beta(5+2\delta)}2} + 2t \eps^{\frac{\beta s_0}2}} \norm{\tU}_{\widetilde{L}^\infty\pare{[0,t], H^s}}^2.
\end{equation*}

If 
\begin{equation}
	\label{eq:hypothesis01} \beta\pare{10 + 4\delta + 4s_0} < 1
\end{equation} 
and if
\begin{equation}
	\label{eq:hypothesis02} C \;\mathcal{C}(U_0^\eps) \, \pare{T_0^{\frac{3}4} \eps^{\frac{1}4 - \frac{\beta(5+2\delta)}2} + 2T_0 \eps^{\frac{\beta s_0}2}} \leqslant \frac{1}2
\end{equation}
then there exists $\eps_2 \in ]0,\eps_1]$ such that, for any $0 < \eps \leqslant \eps_2$ and for any $0 < t < T_0$, we have
\begin{equation*}
	\frac{1}2 \norm{\tU}_{\widetilde{L}^\infty([0,t],H^s)}^2 \leqslant \eps^{\beta s_0} + C \;\mathcal{C}(U_0^\eps) \, t^{\frac{3}4} \eps^{\frac{1}4 - \frac{\beta(5+2\delta)}2} < 2\eps^{\beta s_0},
\end{equation*}
and so,
\begin{equation*}
	\norm{\tU}_{\widetilde{L}^\infty([0,t],H^s)} < 2 \eps^{\frac{\beta s_0}2}. 
\end{equation*}
Thus, the solution $\tU$ exists at least up to a time $T_0 > 0$ satisfies \eqref{eq:hypothesis02}. We remark that, if \eqref{eq:hypothesis01} is satisfied, then there exists $\eps_3 \in ]0,1[$ such that \eqref{eq:hypothesis02} is satisfied, for any $0 < \eps \leqslant \eps_3$. Setting $$\eps^\star = \min\set{\eps_2,\eps_3}, \quad \alpha = \min\set{\frac{1}4 - \frac{\beta(5+2\delta)}2\;,\; \frac{\beta s_0}2} > 0 \quad \mbox{and} \quad \overline{C}  = \frac{1}{6C},$$ we deduce from \eqref{eq:hypothesis01} and \eqref{eq:hypothesis02} that, for any $\eps \in ]0,\eps^\star]$,
\begin{equation*}
	T^\star_\eps \geqslant T_0 \geqslant \overline{C} \mathcal{C}(U_0^\eps)^{-1} \eps^{-\alpha}.
\end{equation*}
Theorem \ref{eq:hi_freq} is proved. \hfill $\square$

\appendix

\section{Estimates on the bilinear terms}

In this appendix, we prove important estimates on the bilinear term, which allow to prove Lemma \ref{le:bi_bootstrap}. First of all, we prove the following lemma

\begin{lemma}
	\label{le:biproduct}
	Let $i \in \set{1,2,3}$ and $\dd_i = \frac{\dd}{\dd x_i}$. For any $s > \frac{5}2$ and for all functions $u$, $v$ and $w$ in $\Hs$, we have
	\begin{equation}
		\label{eq:biproduct}
		\int_0^t \abs{\psca{\tq (w(\tau) \,\dd_i u(\tau)) \;\big|\; \tq u(\tau)}_{L^2}} \, d\tau \leqslant C b_q 2^{-2qs} \norm{w}_{\widetilde{L}^\infty([0,t],H^s)} \norm{u}_{\widetilde{L}^2([0,t],H^s)}^2,  
	\end{equation}
	and
	\begin{multline}
		\label{eq:biproductuv}
		\int_0^t \abs{\psca{\tq (w(\tau) \,\dd_i u(\tau)) \;\big|\; \tq v(\tau)}_{L^2} + \psca{\tq (w(\tau) \,\dd_i v(\tau)) \;\big|\; \tq u(\tau)}_{L^2}} \, d\tau \\ 
		\leqslant C b_q 2^{-2qs} \norm{w}_{\widetilde{L}^\infty([0,t],H^s)} \norm{u}_{\widetilde{L}^2([0,t],H^s)} \norm{v}_{\widetilde{L}^2([0,t],H^s)},  
	\end{multline}
	where $b_q$ is a summable sequence such that $\sum_q b_q = 1$.
\end{lemma}

\begin{proof}

We will only prove Estimate \eqref{eq:biproductuv}. Estimate \eqref{eq:biproduct} can be obtained from \eqref{eq:biproductuv} by choosing $u=v$. Applying the Bony decomposition as in \eqref{eq:Bonydec} to the products $w \dd_i  u$ and $w \dd_i  v$, we can write
\begin{equation}
	\label{eq:bony_dec}
	\int_0^t \abs{\psca{\tq (w(\tau) \,\dd_i u(\tau)) \;\big|\; \tq v(\tau)}_{L^2} + \psca{\tq (w(\tau) \,\dd_i v(\tau)) \;\big|\; \tq u(\tau)}_{L^2}} d\tau \leqslant I_A + I_B,
\end{equation}
where
\begin{align*}
	I_A &= \int_0^t \Big|\Big\langle\Delta_q \sloin S_{q'+2}\pare{\dd_i  u}\Delta_{q'}w \;\Big\vert\; \Delta_q v \Big\rangle + \Big\langle\Delta_q \sloin S_{q'+2}\pare{\dd_i v} \Delta_{q'} w \;\Big\vert\; \Delta_q u \Big\rangle\Big| (\tau) \, d\tau\\
	I_B &= \int_0^t \Big|\Big\langle \Delta_q \spres S_{q'-1} w \, \dd_i  \Delta_{q'} u \;\Big\vert\; \Delta_q v \Big\rangle + \Big\langle \Delta_q \spres S_{q'-1} w \, \dd_i  \Delta_{q'} v \;\Big\vert\; \Delta_q u \Big\rangle \Big| (\tau) \, d\tau.
\end{align*}

Since $S_{q'+2}$ continuously maps $\Linfty$ to $\Linfty$, using Lemma \ref{le:DeltaqHsbis}, H\"older inequality and the Sobolev inclusion $\Hs \hookrightarrow W^{1,\infty}(\R^3)$, we get
\begin{align}
	\label{bound_Iq4} I_A &\leqslant \sumi \norm{S_{q'+2} \dd_i u}_{L^2([0,t],L^\infty)} \left\| \Delta_{q'} w \right\|_{L^\infty([0,t],L^2)} \norm{\tq v}_{L^2([0,t],L^2)}\\
	& \qquad \qquad + \sumi \norm{S_{q'+2} \dd_i v}_{L^2([0,t],L^\infty)} \left\| \Delta_{q'} w \right\|_{L^\infty([0,t],L^2)} \norm{\tq u}_{L^2([0,t],L^2)} \notag\\
	&\leqslant  C \, b_q \, 2^{-2qs} \norm{w}_{\widetilde{L}^\infty([0,t],H^s)} \norm{u}_{\widetilde{L}^2([0,t],H^s)}^2 \notag
\end{align}
where  
$$\set{b_q}_q = \set{\ds\sumi 2^{-\left( q'-q \right) s } c_{q'}(w) \pare{c_q(u) + c_q(v)}}_q \in \ell^1,$$ using Young convolution inequality and the fact that $\set{c_q(u)}_q$, $\set{c_q(v)}_q$ and $\set{c_{q'}(w)}_{q'}$ are square-summable sequences.

To estimate the second term $I_B$ of \eqref{eq:bony_dec}, we decompose $I_B$ as in \cite{CL92} or in \cite{CDGG2}
\begin{equation*}
	I_B \leqslant I_{B1} + I_{B2} + I_{B3},
\end{equation*}
where
\begin{align*}
	I_{B1} &= \int_0^t \abs{\psca{\Sq w \; \tq \dd_i  u \;\big\vert\; \tq v} + \psca{\Sq w \; \tq \dd_i v \;\big\vert\; \tq u}} (\tau) \, d\tau\\
	I_{B2} &= \int_0^t \sumf \abs{\psca{\left( \Sq-S_{q'-1} \right) w \; \Delta_{q'} \dd_i  u \;\big|\; \tq v} + \psca{\left( \Sq-S_{q'-1} \right) w \; \Delta_{q'} \dd_i  v \;\big|\; \tq u}} (\tau) \, d\tau\\
	I_{B3} &= \int_0^t \sumf \abs{\psca{ \left[ \tq , S_{q'-1} w \right]\Delta_{q'} \dd_i u \;\big|\; \tq v} + \psca{ \left[ \tq , S_{q'-1} w \right]\Delta_{q'} \dd_i v \;\big|\; \tq u}} (\tau) \, d\tau.
\end{align*}
Recalling that $S_q$ continuously maps $\Linfty$ to $\Linfty$, an integration by parts, H\"older inequality, Lemma \ref{le:DeltaqHsbis} and the Sobolev inclusion $\Hs \hookrightarrow W^{1,\infty}(\R^3)$ give
\begin{align}
	\label{bound_Iq1} I_{B1} &= \int_0^t \abs{\Sq (\dd_i  w(\tau)) \; \tq u(\tau) \; \tq v(\tau)} d\tau\\ 
	&\leqslant \norm{\Sq \dd_i  w}_{L^\infty([0,t],L^\infty)} \norm{\tq u}_{L^2([0,t],L^2)} \norm{\tq v}_{L^2([0,t],L^2)} \notag\\ 
	&\leqslant C \, b_q \, 2^{-2qs} \norm{w}_{L^\infty([0,t],H^s)} \norm{u}_{\widetilde{L}^2([0,t],H^s)} \norm{v}_{\widetilde{L}^2([0,t],H^s)}, \notag
\end{align}
where $\set{b_q}_q = \set{c_q(u) c_q(v)}_q$ is a summable sequence. For $I_{B2}$, we remark that $\Sq-S_{q'-1}$ does not contains low frequencies and continuously maps $\Linfty$ to $\Linfty$. Then, using Bernstein lemma \ref{lemma:Bernstein} and H\"older inequality, we obtain the same estimates as in \eqref{bound_Iq1}
\begin{align}
	\label{bound_Iq2} I_{B2} &\leqslant C \int_0^t \spres \norm{(\Sq-S_{q'-1})w(\tau)}_{L^\infty} 2^{q'} \norm{\Delta_{q'} u(\tau)}_{L^2} \norm{\Delta_q v(\tau)}_{L^2} d\tau\\
	&\qquad + C \int_0^t \spres \norm{(\Sq-S_{q'-1})w(\tau)}_{L^\infty} 2^{q'} \norm{\Delta_{q'} v(\tau)}_{L^2} \norm{\Delta_q u(\tau)}_{L^2} d\tau \notag\\
	&\leqslant C \spres \norm{(\Sq-S_{q'-1}) \dd_i  w}_{L^\infty([0,t],L^\infty)} \norm{\Delta_{q'} u}_{L^2([0,t],L^2)} \norm{\Delta_q v}_{L^2([0,t],L^2)} \notag \\
	&\qquad + C \spres \norm{(\Sq-S_{q'-1}) \dd_i  w}_{L^\infty([0,t],L^\infty)} \norm{\Delta_{q'} v}_{L^2([0,t],L^2)} \norm{\Delta_q u}_{L^2([0,t],L^2)} \notag\\
	&\leqslant C \, b_q \, 2^{-2qs} \norm{w}_{L^\infty([0,t],H^s)} \norm{u}_{\widetilde{L}^2([0,t],H^s)} \norm{v}_{\widetilde{L}^2([0,t],H^s)} \notag
\end{align}
where
$$\set{b_q}_q = \set{\spres 2^{-(q'-q) s } \pare{c_q(v) c_{q'}(u) + c_q(u) c_{q'}(v)}}_q \in \ell^1.$$
Finally, for the term $I_{B3}$, H\"older inequality and Lemma \ref{commutator estimate} yield
\begin{align*}
	I_{B3} &\leqslant C \sumf 2^{-q} \left\|S_{q'-1} \nabla  w  \right\|_{L^\infty([0,t],L^\infty)} \left\| \Delta_{q'} \dd_i  u \right\|_{L^2([0,t],L^2)} \left\| \tq v \right\|_{L^2([0,t],L^2)}\\
	& \qquad + C \sumf 2^{-q} \left\|S_{q'-1} \nabla  w  \right\|_{L^\infty([0,t],L^\infty)} \left\| \Delta_{q'} \dd_i v \right\|_{L^2([0,t],L^2)} \left\| \tq u \right\|_{L^2([0,t],L^2)}.
\end{align*}
Using the fact that $S_q$ continuously maps $\Linfty$ to $\Linfty$, Bernstein lemma \ref{lemma:Bernstein} and Estimate \eqref{eq:DeltaqHs}, we have
\begin{align}
	\label{bound_Iq3} I_{B3}  &\leqslant C \sumf 2^{q'-q} \norm{S_{q'-1} \nabla  w}_{L^\infty([0,t],L^\infty)} \norm{\Delta_{q'} u}_{L^2([0,t],L^2)} \norm{\tq v}_{L^2([0,t],L^2)}\\
	&\qquad + C \sumf 2^{q'-q} \norm{S_{q'-1} \nabla  w}_{L^\infty([0,t],L^\infty)} \norm{\Delta_{q'} v}_{L^2([0,t],L^2)} \norm{\tq u}_{L^2([0,t],L^2)} \notag\\
	&\leqslant C \, b_q \, 2^{-2qs} \norm{w}_{L^\infty([0,t],H^s)} \norm{u}_{\widetilde{L}^2([0,t],H^s)} \norm{v}_{\widetilde{L}^2([0,t],H^s)}, \notag
\end{align}
where
$$\set{b_q}_q = \set{\spres 2^{-(q'-q)(s-1) } \pare{c_{q'}(u) c_q(v) + c_{q'}(v) c_q(u)}}_q \in \ell^1.$$

\noindent Inserting \eqref{bound_Iq4}--\eqref{bound_Iq3} into \eqref{eq:bony_dec}, we deduce Estimate \eqref{eq:biproductuv}.
\end{proof}

%\bigskip

In order to prove Lemma \ref{le:bi_bootstrap}, we also need the following estimates when the bilinear term contains functions whose Fourier transform is localized in $\Crr$ (see \eqref{eq:CrR} for the definition of $\Crr$).

\begin{lemma}
	\label{le:biprod_cut2}
	Let $T > 0$, $i \in \set{1,2,3}$, $\dd_i = \frac{\dd}{\dd x_i}$ and $\bU$ be the solution of the cut-off linear system \eqref{eq:free_wave}. For any $s > \frac{5}2$, for all functions $v$ and $w$ in $\widetilde{L}^\infty\pare{[0,T],\Hs}$, for any component $\overline{u}^j$ of $\bU$, $j \in \set{1,2,3,4}$, and for any $0 < t \leqslant T$, we have
	\begin{multline}
		\label{eq:biprod_cut2}
		\int_0^t \abs{\psca{\tq \pare{v(\tau) \, \dd_i \overline{u}^j(\tau)} \;\big|\; \tq w(\tau)}_{L^2}} d\tau\\ 
		\leqslant C \;\mathcal{C}(U_0^\eps) \, R^{\frac{5+2\delta}2} t^{\frac{3}4} \eps^{\frac{1}4} b_q 2^{-2qs} \norm{v}_{\widetilde{L}^\infty\pare{[0,t], H^s}} \norm{w}_{\widetilde{L}^\infty\pare{[0,t], H^s}}, \qquad
	\end{multline}
	where $\mathcal{C}(U_0^\eps)$ is defined in \eqref{eq:C0} and $b_q$ is a summable sequence such that $\sum_q b_q = 1$.
\end{lemma}

\begin{lemma}
	\label{le:biprod_cut1}
	Let $T > 0$, $i \in \set{1,2,3}$, $\dd_i = \frac{\dd}{\dd x_i}$ and $\bU$ be the solution of the cut-off linear system \eqref{eq:free_wave}. For any $s > \frac{5}2$, for all functions $v$ and $w$ in $\widetilde{L}^\infty\pare{[0,T],\Hs}$, for any component $\overline{u}^j$ of $\bU$, $j \in \set{1,2,3,4}$, and for any $0 < t \leqslant T$, we have
	\begin{equation}
		\label{eq:biprod_cut1a}
		\int_0^t \abs{\psca{\tq \pare{\overline{u}^j(\tau) \, \dd_i v(\tau)} \;\big|\; \tq v(\tau)}_{L^2}} d \tau \\ 
		\leqslant C \;\mathcal{C}(U_0^\eps) \, R^{\frac{5+2\delta}2} t^{\frac{3}4} \eps^{\frac{1}4} b_q 2^{-2qs} \norm{v}_{\widetilde{L}^\infty\pare{[0,t], H^s}}^2,
	\end{equation}
	and
	\begin{multline}
		\label{eq:biprod_cut1b}
		\int_0^t \abs{\psca{\tq \pare{\overline{u}^j(\tau) \, \dd_i v(\tau)} \;\big|\; \tq w(\tau)}_{L^2} + \psca{\tq \pare{\overline{u}^j(\tau) \, \dd_i w(\tau)} \;\big|\; \tq v(\tau)}_{L^2} } d \tau \\ 
		\leqslant C \;\mathcal{C}(U_0^\eps) \, R^{\frac{5+2\delta}2} t^{\frac{3}4} \eps^{\frac{1}4} b_q 2^{-2qs} \norm{v}_{\widetilde{L}^\infty\pare{[0,t], H^s}} \norm{w}_{\widetilde{L}^\infty\pare{[0,t], H^s}},
	\end{multline}
	where $\;\mathcal{C}(U_0^\eps)$ is defined in \eqref{eq:C0} and $b_q$ is a summable sequence such that $\sum_q b_q = 1$.
\end{lemma}

%\bigskip

\noindent \textit{Proof of Lemma \ref{le:biprod_cut2}.} We apply the same Bony decomposition into paraproducts and remainders as in \eqref{eq:bony_dec} and we have
\begin{equation}
	\label{eq:bony_cut}
	\int_0^t \abs{\psca{\tq \pare{v(\tau) \, \dd_i \overline{u}^j(\tau)} \;\big|\; \tq w(\tau)}_{L^2}} d\tau \leqslant J_A + J_B,
\end{equation}
where
\begin{align*}
	J_A &= \int_0^t \Big|\Big\langle\Delta_q \sloin S_{q'+2}\pare{\dd_i \overline{u}^j(\tau)} \Delta_{q'} v(\tau) \;\Big\vert\; \Delta_q w(\tau)\Big\rangle_{L^2}\Big| \, d\tau\\
	J_B &= \int_0^t \Big|\Big\langle\Delta_q \spres S_{q'-1} v(\tau) \Delta_{q'} \dd_i \overline{u}^j(\tau) \;\Big\vert\; \Delta_q w(\tau)\Big\rangle_{L^2}\Big| \, d\tau.
\end{align*}

For the term $J_A$, Lemma \ref{le:DeltaqHsbis} and similar estimates as in \eqref{bound_Iq4} imply
\begin{align*}
	J_A &\leqslant \sloin \norm{S_{q'+2} \, \dd_i \overline{u}^j}_{L^1([0,t],L^\infty)} \norm{\Delta_{q'} v}_{L^\infty([0,t],L^2)} \norm{\tq v}_{L^\infty([0,t],L^2)}\\
	&\leqslant CRt^{\frac{3}4} \norm{\bU}_{L^4\pare{[0,t], L^\infty}} 2^{-2qs} \pare{\sloin 2^{-(q'-q)s} c_{q'}(v) c_q(w)} \norm{v}_{\widetilde{L}^\infty\pare{[0,t], H^s}} \norm{w}_{\widetilde{L}^\infty\pare{[0,t], H^s}}
\end{align*} 
Using Estimate \eqref{eq:Strichartz4} and fixing $0 < r = R^{-\delta}$, $\delta > 0$, we have
\begin{equation}
	\label{eq:JA} J_A \leqslant C \;\mathcal{C}(U_0^\eps) R^{\frac{5+2\delta}2} t^{\frac{3}4} \eps^{\frac{1}4} b_q 2^{-2qs} \norm{v}_{\widetilde{L}^\infty\pare{[0,t], H^s}} \norm{w}_{\widetilde{L}^\infty\pare{[0,t], H^s}},
\end{equation}
where
$$\set{b_q}_q = \set{\ds\sumi 2^{-(q'-q) s } c_{q'}(v) c_q(w)}_q \in \ell^1.$$

The term $J_B$ is a little more difficult to estimate. Using H\"older inequality and the fact that $S_{q'-1}$ continuously maps $\Hs$ into $\2$, we have
\begin{align*}
	J_B &\leqslant \spres \norm{S_{q'-1} v}_{L^\infty\pare{[0,t], L^2}} \norm{\Delta_{q'} \dd_i \overline{u}^j}_{L^1\pare{[0,t], L^\infty}} \norm{\tq w}_{L^\infty\pare{[0,t], L^2}}\\
	&\leqslant C Rt^{\frac{3}4} \spres \norm{\Delta_{q'} \bU}_{L^4\pare{[0,t], L^\infty}} \norm{v}_{L^\infty\pare{[0,t], H^s}} \norm{\tq w}_{L^\infty\pare{[0,t], L^2}}.
\end{align*}
Estimate \eqref{eq:Strichartz4} implies
\begin{align}
	\label{eq:JB} J_B &\leqslant C R^{\frac{5+2\delta}2} t^{\frac{3}4} \eps^{\frac{1}4} \spres \norm{\Delta_{q'} \Prr U_0}_{L^2} \norm{v}_{L^\infty\pare{[0,t], H^s}} \norm{\tq w}_{L^\infty\pare{[0,t], L^2}}\\
	&\leqslant C R^{\frac{5+2\delta}2} t^{\frac{3}4} \eps^{\frac{1}4} b_q 2^{-2qs} \norm{U_0}_{H^s} \norm{v}_{L^\infty\pare{[0,t], H^s}} \norm{w}_{\widetilde{L}^\infty\pare{[0,t], H^s}} \notag\\
	&\leqslant C \;\mathcal{C}(U_0^\eps) R^{\frac{5+2\delta}2} t^{\frac{3}4} \eps^{\frac{1}4} b_q 2^{-2qs} \norm{v}_{\widetilde{L}^\infty\pare{[0,t], H^s}} \norm{w}_{\widetilde{L}^\infty\pare{[0,t], H^s}}, \notag
\end{align}
where
$$\set{b_q}_q = \set{\ds\spres 2^{-(q'-q) s } c_{q'}(U_0^\eps) c_q(w)}_q \in \ell^1.$$
Putting \eqref{eq:JA} and \eqref{eq:JB} into \eqref{eq:bony_cut}, we deduce Estimate \eqref{eq:biprod_cut2}. \hfill $\square$

%\bigskip

\noindent \textit{Proof of Lemma \ref{le:biprod_cut1}.} As in the proof of Lemma \ref{le:biproduct}, we will only prove Estimate \eqref{eq:biprod_cut1b}. Estimate \eqref{eq:biprod_cut1a} will follow if we choose $v=w$. Applying the Bony decomposition into paraproducts and remainders, we have
\begin{equation}
	\label{eq:bony_cut1}
	\int_0^t \abs{\psca{\tq \pare{\overline{u}^j(\tau) \, \dd_i v(\tau)} \;\big|\; \tq w(\tau)}_{L^2} + \psca{\tq \pare{\overline{u}^j(\tau) \, \dd_i w(\tau)} \;\big|\; \tq v(\tau)}_{L^2} } d \tau \leqslant K_A + K_B,
\end{equation}
where
\begin{align*}
	K_A &= \int_0^t \Big\vert \Big\langle \Delta_q \sloin S_{q'+2}\pare{\dd_i v} \Delta_{q'} \overline{u}^j \;\Big\vert\; \Delta_q w \Big\rangle + \Big\langle \Delta_q \sloin S_{q'+2}\pare{\dd_i w} \Delta_{q'} \overline{u}^j \;\Big\vert\; \Delta_q v \Big\rangle \Big\vert (\tau) \, d\tau\\
	K_B &= \int_0^t \Big\vert \Big\langle \Delta_q \spres S_{q'-1} \overline{u}^j \Delta_{q'} \dd_i v \;\Big\vert\; \Delta_q w \Big\rangle + \Big\langle \Delta_q \spres S_{q'-1} \overline{u}^j \Delta_{q'} \dd_i w \;\Big\vert\; \Delta_q v \Big\rangle \Big\vert (\tau) \, d\tau.
\end{align*}

The term $K_A$ can be bounded by similar estimates as we did for $J_B$ in \eqref{eq:JB}
\begin{align}
	\label{eq:KA} K_A &\leqslant \sloin \norm{S_{q'+2} v}_{L^\infty([0,t],L^2)} \norm{\Delta_{q'} \dd_i \overline{u}^j}_{L^1([0,t],L^\infty)} \norm{\tq w}_{L^\infty([0,t],L^2)} \\
	&\qquad + \sloin \norm{S_{q'+2} w}_{L^\infty([0,t],L^2)} \norm{\Delta_{q'} \dd_i \overline{u}^j}_{L^1([0,t],L^\infty)} \norm{\tq v}_{L^\infty([0,t],L^2)} \notag\\
	&\leqslant C R^{\frac{5+2\delta}2} t^{\frac{3}4} \eps^{\frac{1}4} \sloin \norm{\Delta_{q'} \Prr U_0}_{L^2} \norm{v}_{L^\infty\pare{[0,t], H^s}} \norm{\tq w}_{L^\infty\pare{[0,t], L^2}} \notag\\
	&\qquad + C R^{\frac{5+2\delta}2} t^{\frac{3}4} \eps^{\frac{1}4} \sloin \norm{\Delta_{q'} \Prr U_0}_{L^2} \norm{w}_{L^\infty\pare{[0,t], H^s}} \norm{\tq v}_{L^\infty\pare{[0,t], L^2}} \notag\\
	&\leqslant C \;\mathcal{C}(U_0^\eps) R^{\frac{5+2\delta}2} t^{\frac{3}4} \eps^{\frac{1}4} b_q 2^{-2qs} \norm{v}_{\widetilde{L}^\infty\pare{[0,t], H^s}} \norm{w}_{\widetilde{L}^\infty\pare{[0,t], H^s}}, \notag
\end{align} 
where
$$\set{b_q}_q = \set{\ds\sloin 2^{-(q'-q) s } c_{q'}(U_0^\eps) \pare{c_q(v) + c_q(w)}}_q \in \ell^1.$$

The term $K_B$ is more difficult to estimate because we can not simply commute $S_{q'-1}$ and $\dd_i$. So, we use the same method as for the term $I_B$ of \eqref{eq:bony_dec} and we decompose
\begin{equation*}
	K_B \leqslant K_{B1} + K_{B2} + K_{B3},
\end{equation*}
where
\begin{align*}
	K_{B1} &= \int_0^t \abs{\psca{\Sq \overline{u}^j \; \tq \dd_i v \;\big\vert\; \tq w} + \psca{\Sq \overline{u}^j \; \tq \dd_i w \;\big\vert\; \tq v}} (\tau) \, d\tau\\
	K_{B2} &= \int_0^t \sumf \abs{\psca{\left( \Sq-S_{q'-1} \right) \overline{u}^j \; \Delta_{q'} \dd_i v \;\big|\; \tq w} + \psca{\left( \Sq-S_{q'-1} \right) \overline{u}^j \; \Delta_{q'} \dd_i w \;\big|\; \tq v}} (\tau) \, d\tau\\
	K_{B3} &= \int_0^t \sumf \abs{\psca{ \left[ \tq , S_{q'-1} \overline{u}^j \right]\Delta_{q'} \dd_i v \;\big|\; \tq w} + \psca{ \left[ \tq , S_{q'-1} \overline{u}^j \right]\Delta_{q'} \dd_i w \;\big|\; \tq v}} (\tau) \, d\tau.
\end{align*}
For $K_{B1}$, performing an integration by parts, we have
\begin{align}
	\label{eq:KB1} K_{B1} &= \int_0^t \abs{\Sq (\dd_i \overline{u}^j(\tau)) \; \tq v(\tau) \; \tq w(\tau)} \, d\tau\\ 
	&\leqslant \norm{\Sq \dd_i \bU}_{L^1([0,t],L^\infty)} \norm{\tq v}_{L^\infty([0,t],L^2)} \norm{\tq w}_{L^\infty([0,t],L^2)} \notag\\ 
	&\leqslant C \;\mathcal{C}(U_0^\eps) R^{\frac{5+2\delta}2} t^{\frac{3}4} \eps^{\frac{1}4} b_q 2^{-2qs} \norm{v}_{\widetilde{L}^\infty\pare{[0,t], H^s}} \norm{w}_{\widetilde{L}^\infty\pare{[0,t], H^s}}, \notag
\end{align}
where $\set{b_q}_q = \set{c_q(v) c_q(w)}_q$ is a summable sequence. For $K_{B2}$, we have
\begin{align}
	\label{eq:KB2} K_{B2} &\leqslant C \int_0^t \spres \norm{(\Sq-S_{q'-1}) \overline{u}^j(\tau)}_{L^\infty} 2^{q'} \norm{\Delta_{q'} v(\tau)}_{L^2} \norm{\Delta_q w(\tau)}_{L^2} d\tau\\
	&\qquad + C \int_0^t \spres \norm{(\Sq-S_{q'-1}) \overline{u}^j(\tau)}_{L^\infty} 2^{q'} \norm{\Delta_{q'} w(\tau)}_{L^2} \norm{\Delta_q v(\tau)}_{L^2} d\tau \notag\\
	&\leqslant C \spres \norm{(\Sq-S_{q'-1}) \nabla \bU}_{L^1([0,t],L^\infty)} \norm{\Delta_{q'} v}_{L^\infty([0,t],L^2)} \norm{\Delta_q w}_{L^\infty([0,t],L^2)} \notag \\
	&\qquad + C \spres \norm{(\Sq-S_{q'-1}) \nabla \bU}_{L^1([0,t],L^\infty)} \norm{\Delta_{q'} w}_{L^\infty([0,t],L^2)} \norm{\Delta_q v}_{L^\infty([0,t],L^2)} \notag \\
	&\leqslant  C \;\mathcal{C}(U_0^\eps) R^{\frac{5+2\delta}2} t^{\frac{3}4} \eps^{\frac{1}4} b_q 2^{-2qs} \norm{v}_{\widetilde{L}^\infty\pare{[0,t], H^s}} \norm{w}_{\widetilde{L}^\infty\pare{[0,t], H^s}}, \notag
\end{align}
where
$$\set{b_q}_q = \set{\spres 2^{-(q'-q) s } \pare{c_{q'}(v) c_q(w) + c_{q'}(w) c_q(v)}}_q \in \ell^1.$$
Finally, for $K_{B3}$ we can write
\begin{align}
	\label{eq:KB3} K_{B3}  &\leqslant C \sumf 2^{q'-q} \norm{S_{q'-1} \nabla  \overline{u}^j}_{L^1([0,t],L^\infty)} \norm{\Delta_{q'} v}_{L^\infty([0,t],L^2)} \norm{\tq w}_{L^\infty([0,t],L^2)}\\
	&\qquad + C \sumf 2^{q'-q} \norm{S_{q'-1} \nabla  \overline{u}^j}_{L^1([0,t],L^\infty)} \norm{\Delta_{q'} w}_{L^\infty([0,t],L^2)} \norm{\tq v}_{L^\infty([0,t],L^2)} \notag\\
	&\leqslant C \;\mathcal{C}(U_0^\eps) R^{\frac{5+2\delta}2} t^{\frac{3}4} \eps^{\frac{1}4} b_q 2^{-2qs} \norm{v}_{\widetilde{L}^\infty\pare{[0,t], H^s}} \norm{w}_{\widetilde{L}^\infty\pare{[0,t], H^s}}, \notag
\end{align}
where
$$\set{b_q}_q = \set{\spres 2^{-(q'-q)(s-1) } \pare{c_{q'}(v) c_q(w) + c_{q'}(w) c_q(v)}}_q \in \ell^1.$$
Summing Estimates \eqref{eq:KA} to \eqref{eq:KB3} and putting the obtained result into \eqref{eq:bony_cut1}, we deduce Inequality \eqref{eq:biprod_cut1b} of Lemma \ref{le:biprod_cut1}. \hfill $\square$

\bigskip

\textit{Proof of Lemma \ref{le:bi_bootstrap}} We recall the decomposition of $U^\varepsilon$ as the sum 
\begin{equation*}
	U^\varepsilon= \bU + \tU,
\end{equation*}
then, we can write
\begin{equation*}
	\int_0^t \abs{\psca{\tq \pare{u^\varepsilon(\tau) \cdot \nabla U^\varepsilon(\tau)} \;\Big|\; \tq \tU(\tau)}_{L^2}} \, d\tau \leqslant A_1 + A_2 + A_3 + A_4,
\end{equation*}
where
\begin{align*}
	A_1 &= \int_0^t \abs{\psca{\tq \pare{\overline{u}^\eps(\tau) \cdot \nabla \bU(\tau)} \;\Big|\; \tq \tU(\tau)}_{L^2}} \, d\tau\\
	A_2 &= \int_0^t \abs{\psca{\tq \pare{\widetilde{u}^\eps(\tau) \cdot \nabla \bU(\tau)} \;\Big|\; \tq \tU(\tau)}_{L^2}} \, d\tau\\
	A_3 &= \int_0^t \abs{\psca{\tq \pare{\overline{u}^\eps(\tau) \cdot \nabla \tU(\tau)} \;\Big|\; \tq \tU(\tau)}_{L^2}} \, d\tau\\
	A_4 &= \int_0^t \abs{\psca{\tq \pare{\widetilde{u}^\eps(\tau) \cdot \nabla \tU(\tau)} \;\Big|\; \tq \tU(\tau)}_{L^2}} \, d\tau.
\end{align*}
Using Lemma \ref{le:biprod_cut2}, we have
\begin{align*}
	A_1 &\leqslant C \;\mathcal{C}(U_0^\eps) \, b_q 2^{-2qs} R^{\frac{5+2\delta}2} t^{\frac{3}4} \eps^{\frac{1}4} \norm{\tU}_{\widetilde{L}^\infty\pare{[0,t], H^s}}\\
	A_2 &\leqslant C \;\mathcal{C}(U_0^\eps) \, b_q 2^{-2qs} R^{\frac{5+2\delta}2} t^{\frac{3}4} \eps^{\frac{1}4} \norm{\tU}_{\widetilde{L}^\infty\pare{[0,t], H^s}}^2.
\end{align*}
For $A_3$, using Lemma \ref{le:biprod_cut1}, we have
\begin{equation*}
	A_3 \leqslant C \;\mathcal{C}(U_0^\eps) \, b_q 2^{-2qs} R^{\frac{5+2\delta}2} t^{\frac{3}4} \eps^{\frac{1}4} \norm{\tU}_{\widetilde{L}^\infty\pare{[0,t], H^s}}^2.
\end{equation*}
Finally, for $A_4$, Lemma \ref{le:biproduct} and the Sobolev embedding $H^s\pare{\RR^3} \hookrightarrow W^{1,\infty}\pare{\RR^3}$, with $s > \frac{5}{2}$, simply yield
\begin{equation*}
	A_4 \leqslant C b_q 2^{-2qs} \norm{\tU}_{\widetilde{L}^\infty\pare{[0,t], H^s}} \norm{\tU}_{\widetilde{L}^2\pare{[0,t], H^s}}^2 \leq C b_q 2^{-2qs} t \norm{\tU}_{\widetilde{L}^\infty\pare{[0,t], H^s}} \norm{\tU}_{\widetilde{L}^\infty\pare{[0,t], H^s}}^2.
\end{equation*}

\bigskip

Now, we can prove \eqref{eq:bi_bootstrap2} exactly in the same way as we do to prove \eqref{eq:bi_bootstrap}. We can decompose the term on the right hand side of \eqref{eq:bi_bootstrap2} as
\begin{multline*}
	\int_0^t \abs{\psca{\tq \pare{b^\varepsilon (\tau) \nabla b^\varepsilon (\tau)} \;\big\vert\; \tq \widetilde{u}^\varepsilon (\tau)}_{L^2} + \psca{\tq \left( b^\varepsilon (\tau) \dive u^\varepsilon (\tau) \right) \;\big\vert\; \tq \widetilde{b}^\varepsilon (\tau)}_{L^2} } \, d\tau\\ \leqslant A'_1 + A'_2 + A'_3 + A'_4,
\end{multline*}
where 
\begin{align*}
	A'_1= & \int_0^t \abs{\psca{\tq \pare{\overline{b}^\varepsilon \left( \tau \right) \nabla \overline{b}^\varepsilon \left( \tau \right) } \;\big\vert\; \tq \widetilde{u}^\varepsilon (\tau)}_{L^2} + \psca{\tq \left( \overline{b}^\varepsilon(\tau) \dive \overline{u}^\varepsilon(\tau) \right) \;\big\vert\; \tq \widetilde{b}^\varepsilon \left( \tau \right) }_{L^2} } \, d\tau\\
	A'_2= & \int_0^t \abs{\psca{\tq \pare{\widetilde{b}^\varepsilon \left( \tau \right) \nabla \overline{b}^\varepsilon \left( \tau \right) } \;\big\vert\; \tq \widetilde{u}^\varepsilon (\tau)}_{L^2} + \psca{\tq \left( \widetilde{b}^\varepsilon (\tau) \dive \overline{u}^\varepsilon(\tau) \right) \;\big\vert\; \tq \widetilde{b}^\varepsilon \left( \tau \right) }_{L^2} } \, d\tau\\
	A'_3= & \int_0^t \abs{\psca{\tq \pare{\overline{b}^\varepsilon \left( \tau \right) \nabla \widetilde{b}^\varepsilon \left( \tau \right) } \;\big\vert\; \tq \widetilde{u}^\varepsilon (\tau)}_{L^2} + \psca{\tq \left( \overline{b}^\varepsilon (\tau) \dive \widetilde{u}^\varepsilon(\tau) \right) \;\big\vert\; \tq \widetilde{b}^\varepsilon \left( \tau \right) }_{L^2} } \, d\tau\\
	A'_4= & \int_0^t \abs{\psca{\tq \pare{\widetilde{b}^\varepsilon \left( \tau \right) \nabla \widetilde{b}^\varepsilon \left( \tau \right) } \;\big\vert\; \tq \widetilde{u}^\varepsilon (\tau)}_{L^2} + \psca{\tq \left( \widetilde{b}^\varepsilon (\tau) \dive \widetilde{u}^\varepsilon(\tau) \right) \;\big\vert\; \tq \widetilde{b}^\varepsilon \left( \tau \right) }_{L^2} } \, d\tau.
\end{align*}
Using Lemma \ref{le:biprod_cut2}, we have
\begin{align*}
	A'_1 &\leqslant C \;\mathcal{C}(U_0^\eps) \, b_q 2^{-2qs} R^{\frac{5+2\delta}2} t^{\frac{3}4} \eps^{\frac{1}4} \norm{\tU}_{\widetilde{L}^\infty\pare{[0,t], H^s}}\\
	A'_2 &\leqslant C \;\mathcal{C}(U_0^\eps) \, b_q 2^{-2qs} R^{\frac{5+2\delta}2} t^{\frac{3}4} \eps^{\frac{1}4} \norm{\tU}_{\widetilde{L}^\infty\pare{[0,t], H^s}}^2.
\end{align*}
Next, Lemma \ref{le:biprod_cut1} yields
\begin{equation*}
	A'_3 \leqslant C \;\mathcal{C}(U_0^\eps) \, b_q 2^{-2qs} R^{\frac{5+2\delta}2} t^{\frac{3}4} \eps^{\frac{1}4} \norm{\tU}_{\widetilde{L}^\infty\pare{[0,t], H^s}}^2.
\end{equation*}
Finally, Lemma \ref{le:biproduct} and the Sobolev embedding $H^s\pare{\RR^3} \hookrightarrow W^{1,\infty}\pare{\RR^3}$, with $s > \frac{5}{2}$, imply
\begin{equation*}
	A'_4 \leqslant C b_q 2^{-2qs} \norm{\widetilde{u}^\eps}_{\widetilde{L}^\infty\pare{[0,t], H^s}} \norm{\widetilde{b}^\eps}_{\widetilde{L}^2\pare{[0,t], H^s}}^2 \leq C b_q 2^{-2qs} t \norm{\tU}_{\widetilde{L}^\infty\pare{[0,t], H^s}} \norm{\tU}_{\widetilde{L}^\infty\pare{[0,t], H^s}}^2.
\end{equation*}
Lemma \ref{le:bi_bootstrap} is then proved. \hfill $\square$

\end{document}